\documentclass[9pt,a4paper,reqno]{amsart}
\pdfoutput=1
\usepackage{enumitem}
\usepackage[utf8]{inputenc}
\pdfoutput=1
\usepackage{parskip}
\usepackage{mathrsfs}
\usepackage{amsthm}
\usepackage{amsmath}
\usepackage{amsfonts}
\usepackage{amssymb}
\usepackage{graphicx}
\usepackage{color}
\usepackage{graphics}
\usepackage{eepic}
\usepackage{nicefrac}
\usepackage{hyperref}
\usepackage{bigints}
\usepackage{bbm}
\usepackage{array}
\usepackage{caption}
\usepackage[utf8]{inputenc}
\usepackage{amsmath}
\usepackage[nobysame]{amsrefs}

\makeatletter
\@addtoreset{equation}{chapter}
\makeatother
\usepackage{tikzsymbols}
\usepackage{booktabs}
\hypersetup{
 colorlinks   = true,
 urlcolor     = blue,
 linkcolor    = blue,
 citecolor   = red,
 bookmarksopen=true
}

\newcommand{\ignore}[1]{}

\usepackage[colorinlistoftodos]{todonotes}

\newcommand{\be}{\begin{equation}}
\newcommand{\ee}{\end{equation}}

\renewcommand{\Re}{\operatorname{Re}}

\DeclareMathOperator{\Var}{Var}
\newcommand{\C}{{\mathbb{C}}}

\newcommand{\N}{{\mathbb{N}}}

\newcommand{\D}{{\mathbb{D}}}

\newcommand{\E}{{\mathbb{E}}}

\usepackage{mathtools} 

\newcommand{\tildea}[1]{\overset{\sim}{#1}}

\newtheoremstyle{mine}
{\baselineskip}
{\baselineskip}
{\itshape}
{
}
{\bfseries}
{.}
{.5em}
{#1 #2\ifx#3\relax\else~(#3)\fi}

\theoremstyle{mine}

\newtheorem{thm}{Theorem}[section]

\newtheorem{prop}[thm]{Proposition}

\newtheorem{Lemma}[thm]{Lemma}

\newtheorem*{claim*}{Claim}

\newtheorem{remark}[thm]{Remark}

\theoremstyle{definition}

\theoremstyle{remark}

 \renewcommand\epsilon{\varepsilon}

\usepackage{palatino}
 \usepackage{color, xcolor, mathrsfs, graphicx, hyperref, framed}

\newcommand{\mb}{\mathbb}

\usepackage[italian,english]{babel}
\usepackage[autostyle]{csquotes}  

\title[Number of components of random lemniscates]{Asymptotics of the number of components of random polynomial lemniscates}

\author{Subhajit Ghosh, Koushik Ramachandran, Atul Shekhar }

\address[Subhajit Ghosh]{Department of Mathematics, Bar-Ilan University, Ramat Gan, 5290002, Israel
}
\email{ghoshsu1@biu.ac.il}

\address[Koushik Ramachandran]{Tata Institute of Fundamental Research, Centre For Applicable Mathematics, Post Bag No 6503, GKVK Post Office, Sharada Nagar, Chikkabommsandra, Bangalore
560065, India.}
\email{koushik@tifrbng.res.in}

\address[Atul Shekhar]{Tata Institute of Fundamental Research, Centre For Applicable Mathematics, Post Bag No 6503, GKVK Post Office, Sharada Nagar, Chikkabommsandra, Bangalore
560065, India.}
\email{atul@tifrbng.res.in}

\begin{document}

\begin{abstract}
    Consider a sequence of random polynomials $P_n(z) = \prod_{k=1}^{n}(z - X_k)$, where $\{X_k\}_k$ are i.i.d. random variables distributed uniformly on the unit disc $\D$. Let $\Lambda_n = \{z \in \C:  |P_n(z)| < 1\}$ be the lemniscate of $P_n$, and let $\mathscr{C}(\Lambda_n)$ be the number of connected components of $\Lambda_n$. In this paper, we prove that $\lim_{n\to\infty}\frac{\E[\mathscr{C}(\Lambda_n)]}{\sqrt{n}}= \gamma$, and identify the constant $\gamma$.\\ 
\end{abstract}

\maketitle

\section{Introduction}\label{sec1}
In this article we continue the study of the topology of  random polynomial lemniscates that was initiated in \cite{Ghosh-numberofcomponents}. The setting of the problem is as follows. Let $\{X_k\}_{k\geq 1 }$ be an i.i.d. sequence of random variables distributed uniformly in the unit disc $\mathbb{D}$. For each $n,$ consider the random polynomial
\[P_n(z) = (z-X_1)(z-X_2)...(z-X_n).\]

The unit lemniscate of $P_n$ is the random open set defined by \[\Lambda_n = \{z \in \C:  |P_n(z)| < 1\}.\] 
Let $\mathscr{C}(\Lambda_n)$ be the number of connected components of $\Lambda_n$. As a first step towards understanding the random variable $\mathscr{C}(\Lambda_n),$ we would like to know the behavior of the expectation $\E(\mathscr{C}(\Lambda_n))$ for large $n$. Taking logarithms, we note that $\Lambda_n = \{z: \log|P_n(z)| < 0\}$. The advantage with rewriting $\Lambda_n$ this way is that for each $z\in\C$, the quantity $\log|P_n(z)|$ is a sum of $n$ i.i.d. random variables. This lends itself to a lot of probabilistic analysis in the form of the strong law and central limit theorems. Using this connection, the first named author (S.G) in \cite{Ghosh-numberofcomponents}, proved the following Theorem.

\begin{thm}[Ghosh, $2024$]\label{Basic result}
There exist positive constants $C_1$ and $ C_2$ such that for all large $n$ we have
\[C_1\sqrt{n}\leq \E(\mathscr{C}(\Lambda_n))\leq C_2\sqrt{n}.\]
\end{thm} 

 One can give a crude upper bound for the number of components from deterministic considerations. The maximum principle from complex analysis implies that every connected component of $\Lambda_n$ must contain a root of $P_n$. Hence $\mathscr{C}(\Lambda_n)\leq n$, but we expect that typically it should be of smaller order. Theorem \ref{Basic result} quantifies this belief. Let us first recall the strategy of \cite{Ghosh-numberofcomponents} to prove Theorem \ref{Basic result}. A routine computation gives that $\E(\log|P_n(z)|) = n\left(\frac{|z|^2 - 1}{2}\right)$ for $z\in \mathbb{D}$. This is negative throughout the disc. Next, concentration inequalities applied to the random variable $\log|P_n(z)|$ show that "nearly" all of the unit disc $\D$ is contained in $\Lambda_n$ with very high probability. Hence any small components must lie very close to the unit circle. But for points near the unit circle $\E(\log|P_n(z)|)\approx 0$. The fluctuation of $\log|P_n(z)|$ for these points is governed by the Central Limit Theorem. Indeed, the lower bound in Theorem \ref{Basic result} comes from estimating the number of small components near the unit circle, in conjunction with the Berry-Esseen Theorem. The upper bound is a bit more delicate to handle. A basic tool here is Lemma \ref{components=critical} below, relating $\mathscr{C}(\Lambda_n)$ to the number of critical points of $P_n$ that lie outside $\Lambda_n$. Combining this with a deep result on the pairing phenomenon of roots and critical points of random polynomials \cite{Kabluchko-Seidel} finishes the proof (see Section \ref{sketch} below for more on this).

 \vspace{0.1in}
 
 Once the bounds in Theorem \ref{Basic result} are proven, a natural question is whether it is true that $\lim_{n\to\infty}\frac{\E(\mathscr{C}(\Lambda_n))}{\sqrt{n}}$ exists. Our main result below establishes this and identifies the limiting value.
\begin{thm}\label{main-thm}
Let $\mathscr{C}(\Lambda_n)$ be as above. Then 
\begin{align}\label{mainresult}
       \lim_{n \to \infty} \frac{\mb{E}[\mathscr{C}(\Lambda_n)]}{\sqrt{n}} = \sqrt{\frac{\zeta(2) -1}{\pi}}.
    \end{align}
\end{thm}

Some remarks are in order: 

\begin{enumerate}

    \item The limit appearing in Theorem \ref{main-thm} is dependent on the distribution of $X_1$ in that it is equal to a multiple of $\sqrt{Var(\log(|1-X_1|))}$. In Section \ref{some-comp-last} we explicitly compute $Var(\log(|1-X_1|))$ and it turns out to be $\frac{1}{2} \left(\zeta(2) - 1 \right)$.
 
    \item The $\sqrt{n}$ scaling manifests a phase transition in the behaviour of $\mb{E}[\mathscr{C}(\Lambda_n)]$. If $X_n$ are instead uniformly distributed on the disc $\mathbb{D}_r = \{z \hspace{1mm} \bigl| \hspace{1mm} |z| <r\}$ of radius $r>0$, then $\mb{E}[\mathscr{C}(\Lambda_n)]$ is $\Theta(1)$ for $r<1$, and $\Theta(n)$ for $r>1$, see \cite[Remark~$1.3$]{Ghosh-numberofcomponents} for a brief discussion on this.
    
    \item The convergence in \eqref{mainresult} is delicate to prove. We show that an error bound of $O(\frac{1}{(\log n)^\gamma})$ holds for any $\gamma>0$. It will be interesting to see if this can be improved to a $O(\frac{1}{n^\epsilon})$ bound. 

    \item Our proof crucially uses the assumption that the random variables $X_k$ have a density with respect to the Lebesgue measure. It is also 
    of considerable convenience that random variables $X_k$ are uniformly distributed. We believe that (a version of ) Theorem \ref{main-thm} will hold for more general random variables $X_k$ satisfying appropriate assumptions, but the nature of these assumptions seem to be complicated to write down.

    \item The polynomials $P_n(z)$ with i.i.d. roots can be considered as a toy model for the characterstic function of a random matrix sampled from the Ginibre ensemble. It is well known that the empirical distribution of the eigenvalues of such matrices converges to the uniform distribution on the unit disc (the circular law), see \cite{Tao-Vu-Randommatrices-universality-circularlaw}. With Theorem \ref{main-thm} established, it would be interesting (and natural) to ask how many components does the level set of the characterstic function of a Ginibre ensemble have?

\end{enumerate}

Before we explain the strategy involved in proving Theorem \ref{main-thm}, we briefly recall prior works and relevant literature in the study of random lemniscates and related areas. 

\vspace{0.1in}

\noindent Random polynomials with i.i.d. roots were first considered by Pemantle and Rivin \cite{PemantleRivin} to study a \textit{pairing phenomena} between roots and critical points. This phenomenon asserts that for such polynomials, each root has a nearby critical point. Hence if we have an understanding of the typical distance between a root and its nearest critical point, one can in principle use Taylor's theorem to estimate the corresponding critical value. In this way,  Lemma \ref{components=critical} below acts as a bridge connecting the pairing phenomenon to the study of lemniscate components. Pemantle and Rivin \cite{PemantleRivin} conjectured (and proved it for some special cases) that if roots are sampled as i.i.d. random variables, then the empirical distribution of the critical points of such polynomials converges weakly to the distribution of an individual zero. This was later proved by Kabluchko in full generality \cite{KabluchkoCriticalpointsofrandompolynomia}. A more quantitative version of this pairing phenomena was proven in \cite{Kabluchko-Seidel}: the distance between a critical point and the root nearest to it is of order $O(n^{-\alpha})$ for $1/2< \alpha <1$ while the typical spacing between i.i.d. roots is bigger at $O(n^{-\frac{1}{2}})$. Subsequent works established related results for the zeros of higher derivatives of random polynomials and stronger modes of convergence (see \cites{Angst-Jurgen-Poly-Almostsurebehaviorofthecriticalpointsofrandom, Byun-Sung-Soo-Lee-Reddy-their-higher-derivatives, Michelen-Xuan-Truong-Almostsurebehaviorofthezerosofiteratedderivativesof}).
The pairing phenomena can also be observed in the deterministic setting in the famous Sendov's conjecture, which is now essentially proved in the recent work of Terence Tao \cites{Tao-sendov-conjecture}. For more on this phenomena we refer the reader to the works \cites{HaninBorisPairingofzeros, O'RourkeNoahPairingbetweenzerosandcritical,  O'RourkeNoahOnthelocalpairingbehavior,  Kabluchko-Seidel} and references therein. \\

The number of connected components of random rational lemniscates for rational functions $p/q$, where \(p\) and \(q\) are sampled independently from the complex Kostlan ensemble was studied by Lerario and Lundberg \cite{lerario-lundberg-PLMS16}. They showed that the expected number of connected components grows linearly with the degree \(n\). Later, Kabluchko and Wigman \cite{Kabluchko-Wigman} obtained the exact asymptotics for nodal components in terms of the Nazarov--Sodin constant $c_{NS}$ \cite{Nazarov-Sodin}. For a sequence of Kac polynomials $p_n$ (with i.i.d complex Gaussian coefficients), it was shown by Lundberg and Ramachandran in \cite{lundberg-ramachandran-JLMS17} that $n^{-1}\E(\mathscr{C}_{p_n})\to\ 1$. Hence, in this model, the average number of components is asymptotically the maximal it can be. Another related work is \cite{GhoshRamachandran} where the first two named authors (S.G and K.R) used probabilistic methods to answer some (deterministic) questions of Erd\"{o}s on the maximal number of components when the zeros of the polynomial are constrained to lie on a given compact set. The study of certain other geometrical aspects of random polynomial lemniscates was investigated in \cite{KLR}. We now move on to give a rough sketch of the proof of Theorem \ref{main-thm}.


\subsection{Sketch of the proof of Theorem \ref{main-thm} and organization of the paper}\label{sketch}

It is a well known fact that the number of connected components of $\Lambda_n$ is one more than the number of critical points of $P_n$ which lie outside $\Lambda_n$, see Lemma \ref{components=critical}. By the Gauss-Lucas theorem, all the critical points of $P_n$ lie in $\mb{D}$. Hence, we have to count the number of critical points in $\mb{D}\setminus \Lambda_n$: denote this number as $L_n$. We assert that  $\mb{E}[L_n]\approx n  \mb{E}[Area(\mb{D}\setminus \Lambda_n)] $. A heuristic argument to support this claim can be given as follows. By \emph{the pairing phenomena} alluded to before, we can treat a critical point of $P_n$ as uniformly distributed on $\mb{D}$. Hence, pretending as if critical points and $\Lambda_n$ are independent of each other, this suggests that the probability that a critical point lies in $\mb{D} \setminus\Lambda_n $ is roughly $\frac{1}{\pi}Area(\mb{D} \setminus\Lambda_n)$. Since there are $n-1$ critical points, this suggests that $\mb{E}[L_n]$ has the same asymptotic behaviour as $n \times \mb{E}[Area(\mb{D}\setminus \Lambda_n)]$. Taking this assertion granted for now, we next claim that $\mb{E}[Area(\mb{D}\setminus \Lambda_n)] \sim  C/\sqrt{n}$, which in turn implies Theorem \ref{main-thm}. We verify this claim in two steps. Firstly, we obtain an a priori inradius bound of $\Lambda_n$ stating that $\Lambda_n$ contains a large ball with high probability, see Figure \ref{fig: location of lemniscates} and Proposition \ref{location of the components of the lemniscate: lemma}. This a priori estimate gives a $O(\sqrt{\log n /n})$ bound on $\mb{E}[Area(\mb{D}\setminus \Lambda_n)]$. Secondly, to get rid of the $\sqrt{\log n}$ factor and enhance this to the asymptotic $\mb{E}[Area(\mb{D}\setminus \Lambda_n)] \sim C/\sqrt{n}$, we write $\mb{E}[Area(\mb{D}\setminus \Lambda_n)]$ in the form of an integral representation and invoke the Edgeworth expansion theorem
(c.f. Theorem \ref{edgeworth expansion Theorem} below). 

The main difficulty that arises is to rigorously prove that $\E[L_n]\approx n \mb{E}[Area(\mb{D}\setminus \Lambda_n) $. We were unable to proceed along the lines of the heuristic stated above. Instead we use a Kac-Rice method to obtain an explicit formula for $\mb{E}[L_n]$, see Proposition \ref{Kac-Rice Formula} and Proposition \ref{key-prop}. This yields $\E[L_n] = nM_n + R_n$, where the term $M_n$ is approximately equal to $\mb{E}[Area(\mb{D}\setminus \Lambda_n)]$, and the remainder term $R_n$ is negligible, see Section \ref{analyse-the-kac-rice} for the details. The analysis of $R_n$ is quite delicate. Our key input, Proposition \ref{dev-anti-conc}, combines a large deviation estimate in the heavy tail regime and an anticoncentration/small ball inequality to handle $R_n$. This is done in Section \ref{section:A Large Deviated anticoncentration/small ball estimate} and Section \ref{last-analysis-of-remainder}. 

\subsection{Notations}\label{notations} 
We list below the notations that will be used throughout the paper.
\begin{itemize}
\item $f(n) \lesssim g(n)$ means there exists a constant $C>0$ such that $f(n) \le C g(n)$ for sufficiently large $n$.
\item We denote $f(n) = o(g(n))$ to mean $\lim_{n\to\infty} f(n)/g(n) = 0$, $f(n) = O(g(n))$ to mean $f(n) \lesssim g(n)$, and $f(n) = \Theta(g(n))$ if $f(n) = O(g(n))$ and $g(n) = O(f(n))$.
\item $\mb{E}[\cdots]$ refers to expected value w.r.t. to every random variable present in the expression $\cdots$, whereas $\mb{E}_{i}$ refers to expectation w.r.t. to $X_i$ 
conditional on $\{X_j\}_{j\neq i}$.

\item 
   $P_n(z) = (z-X_1)(z-X_2)...(z-X_n), \quad Q_n(z)= \frac{P_n(z)}{(z-X_1)}$.
   
\begin{align*}
     R_n(z) := \sum_{k=2}^n \frac{1}{(z-X_k)^2}, \hspace{2mm}S_n(z) := \sum_{k=2}^n \frac{1}{z-X_k}, \hspace{2mm}\tildea{S}_n(z) := \sum_{k=3}^n \frac{1}{z-X_k}. 
\end{align*}

\item For a uniform random variable $X_0$ independent of $X_1,X_2,...,X_n$, we write 
\begin{align*}
   R_n := R_n(X_0),  \textrm{\hspace{2mm}}  S_n := S_n(X_0)  \textrm{\hspace{2mm}and\hspace{2mm}} \tildea{S}_n := \tildea{S}_n(X_0). 
\end{align*}
\item Throughout the paper, we denote by C a positive numerical \emph{absolute} constant whose values may vary from line to line.
\end{itemize}

\section{Preliminaries}

\subsection{Connected components and critical points}
\begin{Lemma}\label{components=critical}
Let $\{\beta_j\}_{j=1}^{n-1}$ be the critical points of the random polynomial $P_n$. Then,
\begin{enumerate}[label=(\alph*)]
    \item $\mathscr{C}(\Lambda_n)= 1+ \left|\left\{ j: |P_n(\beta_j)| \geq 1\right \}\right|$. 
    
    \item All the critical points of $P_n$ are simple, i.e. $P_n''(\beta_j) \neq 0$ for all $j$ almost surely. Furthermore, $|P_n(\beta_j)| \neq 1 $ for all $j$  almost surely. In particular, 
    \begin{align*}
           \mathscr{C}(\Lambda_n)= 1+ \left|\left\{ j: |P_n(\beta_j)| > 1\right \}\right|. 
    \end{align*}
\end{enumerate}
\end{Lemma}
\begin{proof}
The claim $(a)$ is well known, see [Lemma $2.8$, \cite{Ghosh-numberofcomponents}] for details. For part $(b)$ suppose that $P_n'(z) = 0$ and $P_n''(z)=0$ has a common root $z=z_0$. Clearly, since roots of $P_n$ are all distinct almost surely, $z_0 \neq X_i$ for all $i$ and $P_n(z_0) \neq 0$. Using 
\[ P_n'(z) = P_n(z) \sum_{k=1}^n \frac{1}{z-X_k}, \textrm{ \hspace{2mm}and \hspace{2mm}} P_n'(z)^2 - P_n(z)P_n''(z)  = P_n(z)^2 \sum_{k=1}^n \frac{1}{(z-X_k)^2}, \]
it follows that for $z=z_0$, \[\sum_{k=1}^n \frac{1}{z_0-X_k} = 0, \textrm{ \hspace{2mm}and \hspace{2mm}} \sum_{k=1}^n \frac{1}{(z_0-X_k)^2} = 0.\]
This implies that 
\begin{equation}\label{X_1=f(S)}
    X_1 = z_0 + \biggl(\sum_{k=2}^n \frac{1}{z_0-X_k}\biggr)^{-1},
\end{equation} 
and \[\biggl(\sum_{k=2}^n \frac{1}{z_0-X_k}\biggr)^2 + \sum_{k=2}^n \frac{1}{(z_0-X_k)^2} = 0. \]
The set of $z_0$ satisfying the above relation forms a finite subset in $\sigma(X_2, X_3,\cdots, X_n)$, call it the set $S$. Then, using \eqref{X_1=f(S)}, we obtain that $X_1 \in f(S)$ where $f(z) = z + \bigl(\sum_{k=2}^n \frac{1}{z-X_k}\bigr)^{-1}$. Since $X_k$ are independent of each other, by first conditioning w.r.t. $X_2, X_3,...,X_n$ and then integrating against the law of $(X_2, X_3,...,X_n)$, we have $\mb{P}[X_1 \in f(S)] =0$ which proves the claim that $P_n'$ and $P_n''$ has no common root almost surely. \\

Similarly, if there exists $z=z_0$ with $P_n'(z_0) =0 $ and $|P_n(z_0)| =1 $, we would obtain \eqref{X_1=f(S)} and $|z_0-X_2||z_0 -X_3|...|z_0-X_n| = |z_0-X_1|^{-1}$. This, in turn, implies that $z_0$ solves the equation 
\[ \bigl|\sum_{k=2}^n \frac{1}{z_0 - X_k}\bigr| =|z_0-X_2||z_0 -X_3|...|z_0-X_n|. \]
This implies that $z_0$ lies in the zero set of a non-constant multivariate polynomial measurable with respect to $\sigma(X_2, X_3,..., X_n)$, call this zero set $Z$. Note that the zero set of a non-constant multivariate polynomial has Lebesgue measure zero. Hence, with $f$ as above, $f(Z)$ has Lebesgue measure zero and $\mb{P}[X_1 \in f(Z)] =0$ which finishes the proof.  
\end{proof}


\subsection{Probabilistic preliminaries}
\begin{thm}\label{edgeworth expansion Theorem}\textbf{(Edgeworth expansion)}\cite[Theorem $1$-Chapter $6$]{Petrov-sumsofindipendentrandomvariables}
    Let $\{Y_i\}_{i \in \N}$ be a sequence of i.i.d. random variables with the common distribution function $F$. Assume that 
    \begin{align*}
        \mb{E}[Y_1]=0, \quad \mb{E}[Y_1^2]= \sigma^2 <\infty,  \quad \textrm{and} \quad  \mb{E}[|Y_1|^3] < \infty.
    \end{align*}
    Let $S_n := \sum_{j=1}^nY_j$ and $\nu(t):=\mb{E}[e^{itY_1}]$ be the characteristic function  of $Y_1$. Then, for $F_n(x):=\mb{P}\left(S_n/\sigma\sqrt{n} \leq x\right)$ and distribution function $\Phi$ of the standard normal random variable, there exists a universal constant $C$ such that for all $x\in \mb{R}$ and $n\geq 1$, 
     \begin{multline}\label{edgeworth expansion}
        \left| F_n(x)- \Phi(x)- \frac{Q_1(x)}{\sqrt{n}}\right| \leq C \Bigg\{ \frac{1}{\sigma^3 \sqrt{n}} \frac{1}{(1+|x|)^3} \int_{|y|\geq \sigma\sqrt{n}(1+|x|)} |y|^3 dF(y)\\[1.5em]
          \hphantom{xxxxxxxxxxx} +\frac{1}{\sigma^4{n}}  \frac{1}{(1+|x|)^4} \int_{|y|< \sigma\sqrt{n}(1+|x|)}  |y|^4 dF(y) \\[1em]
         + \left( \sup_{|t|\geq \delta} |\nu(t)|+\frac{1}{2n}\right)^n  \frac{n^6}{(1+|x|)^4}\Bigg\},
    \end{multline}  
    
    where $Q_1(x) = - \frac{\gamma_3}{6\sigma^3\sqrt{2\pi}} (x^2-1) e^{-x^2/2}$, $\gamma_3=\mb{E}[Y_1^3]$ and $\delta= \frac{\sigma^2}{12 \mb{E}(|Y_1|^3)}$.
\end{thm}

\begin{thm}\label{bobkov et al Theorem}\cite[Theorem 1]{Bobkov-etal}
    Let $X$ be a random variable with density $p$ and variance 
$\sigma^{2} = \mathrm{Var}(X)< \infty$. Assume that $p(x) \leq M$, for all $x\in \mb{R}$ almost everywhere. Then, for the characterstic function $\nu(t) = \mathbb{E}\, e^{itX}$ of $X$, there exists universal constants $C_1, C_2 >0$ such that 

\begin{enumerate}
    \item For all $ |t| \,\geq\, \tfrac{\pi}{4\sigma}$,
    \begin{equation}\label{bobkov eq:bound1}
        |\nu(t)| \;<\; 1 - \frac{C_{1}}{M^{2}\sigma^{2}}.
    \end{equation}

    \item For $0 <  |t| < \tfrac{\pi}{4\sigma}$,
    \begin{equation}\label{bobkov eq:bound2}
        |\nu(t)|\;<\; 1 - \frac{C_{2}t^{2} }{M^{2}} .
    \end{equation}
\end{enumerate}
\end{thm}

\begin{thm}[Bennett's inequality]\cite[Theorem 2.9]{concentration-inequalities-BLMbook}\label{Bennett's inequality}
Let $Y_1,Y_2,...,Y_n$ be independent random variables with finite variances such that  $\forall$ $ i \leq n$, $Y_i \leq b$, for some $b > 0$ almost surely. Let
            \begin{align*}
                S=\sum_{i=1}^n\left(Y_i-\mathbb{E}[Y_i]\right),
            \end{align*}
        and $\nu = \sum_{i=1}^n\mathbb{E}[{Y_i}^2] .$ Then for any $t >0,$ we have
            \begin{align*}
                \mathbb{P}(S > t)\leq \exp\left(-\frac{\nu}{b^2}h\Big(\frac{bt}{\nu}\Big)\right),
            \end{align*}
        where $h(u) = (1 + u) \log(1 + u)- u \textit{, for } u > 0.$
    \end{thm}

   \begin{Lemma}[von Bahr-Esseen inequality]
\cite[Page 9]{Kabluchko-Seidel}, \cite{vonBahr-Esseen} \label{BE-ineq}
   Let $Z_1, Z_2, \dots, Z_n$ be independent, mean-zero random variables. 
Then, for any $1 \leq p \leq 2$, 
\[
    \mathbb{E}\left| \sum_{i=1}^n X_i \right|^p 
    \leq 2 \sum_{i=1}^n \mathbb{E}|X_i|^p .
\] 
\end{Lemma} 

\begin{Lemma} \cite[Lemma~2.4]{Ghosh-numberofcomponents}\label{lemma: p<2 moment bound for uniform on disk}
     Let $X$ be a uniform random variable on the open unit disk $\mathbb{D}$. For $ p<2 $, there exists a constant $C_p$ such that for all $ z \in \mathbb{D}$,
    \begin{align}
            \mb{E}\left[\frac{1}{|z-X|^p}\right] \leq C_p.
        \end{align}
        Furthermore, 
        \begin{equation}\label{mean of 1/z-X}
             \mb{E}\left[\frac{1}{z-X}\right]  = \overline{z}.
        \end{equation}
\end{Lemma}

The detailed computation of \eqref{mean of 1/z-X} is provided in \cite[Example~3.2]{Kabluchko-Seidel}

\subsection{Complex analysis preliminaries.}
\begin{Lemma} [Area Theorem] \cite[Section-$2.4$]{Ahlfors}, \cite[Theorem $3$]{crane-polynomialpreimage} \label{area Lemma}
\begin{enumerate}[label=(\alph*)]
\item Let \( f \) be a univalent (injective holomorphic) function on a domain \( D \subseteq \mathbb{C} \), and let \( K \subset D \) be a measurable set. Then the area of the image \( f(K) \) is given by:
\[
\text{Area}\big(f(K)\big) = \int_K |f'(z)|^2 \, dA(z),
\]
where \( dA(z) = dx \, dy \) denotes the Euclidean area element in \( \mathbb{C} \).
\item Let \( p \) be a polynomial of degree \( n \) over \( \mathbb{C} \), and let \( K \) be any measurable subset of the complex plane. Let \( n(z, p, K) \) be the number of preimages of \( z \) under \( p \) that lie in \( K \) counted with multiplicity. Then,
\begin{align}
    \int_{\mathbb{C}} n(z, p, K)\, dA(z) = \int_K |p'(z)|^2\, dA(z).
\end{align}
\end{enumerate}
\end{Lemma}

\begin{Lemma} \cite[Lemma ~2.2]{KLR}\label{Sup bound of polynomials}
Let $G$ be a bounded Jordan domain with rectifiable boundary, and let $p(z)$ be a polynomial 
of degree $n$. Then there exists a constant $C = C(G) > 0$, and points 
$w_1, w_2, \ldots, w_{Cn^2} \in \partial G$ such that 
\begin{equation} \label{eq:lemma2.2}
    \|p\|_{\partial G} \leq 2 \max_{1 \leq k \leq Cn^2} |p(w_k)|.
\end{equation}
\end{Lemma}



\begin{Lemma}\label{log bound for some integral: lemma}
    Let $B \subset \mb{D} $ and $\alpha \in \mb{C}$. Then there exists an absolute constant $C>0$, such that 
\begin{align}\label{log bound for some integral: equation}
        \bigintsss_B \frac{1}{|w-\alpha|^2} dA(w) \leq C  \log \left(\frac{\sup_{w\in B}|w-\alpha|}{\inf_{w \in B}|w-\alpha|}\right).
    \end{align}
\end{Lemma}

\begin{proof}
    If $\alpha  \in \overline{B}$, then the right-hand side of \eqref{log bound for some integral: equation} in $+\infty$, and this trivially satisfies the inequality. Therefore, assuming that $\alpha \notin \overline{B}$, we express 
\eqref{log bound for some integral: equation} in polar coordinates using the transformation 
$w = \alpha + r e^{i\theta}$ as
\begin{align*}
    \bigintsss_B \frac{1}{|w-\alpha|^2} dA(w)= \bigintsss_{\theta =0}^{2\pi}  \bigintsss_{r: \alpha + r e^{i\theta} \in B} \frac{1}{r^2} rdrd\theta.
\end{align*}
For $z:= \alpha + r e^{i\theta} \in B$, the value of $r=|z-\alpha|$, 
 for any value of $\theta$, the maximal range of $r$ is $\big[ \inf_{z \in B} |z - \alpha|, \; \sup_{z \in B} |z - \alpha| \big].$
Therefore, we obtain
    \begin{align*}
          \bigintsss_B \frac{1}{|w-\alpha|^2} dA(w) \leq \bigintsss_{\theta =0}^{2\pi} \bigintsss_{\inf_{z \in B}|w-\alpha|}^{\sup_{z\in B}|w-\alpha|} \frac{dr}{r} d\theta \leq 2 \pi \log \left(\frac{\sup_{z\in B}|w-\alpha|}{\inf_{z \in B}|w-\alpha|}\right),
    \end{align*}
    which establishes the result.
\end{proof}

\subsection{The Kac-Rice formula.}

The Kac-Rice method for counting roots of a polynomial was introduced by Kac \cites{Kac-I, Kac-II} and Rice \cite{Rice-Mathematicalanalysisofrandomnoise}. Their formula was designed to count the number of real roots for polynomials with i.i.d. coefficients (Kac ensembles). This methodology is nowadays a very standard tool, and it comes in many generalized forms and has been utilized in many applications, see \cite[Chapter 3]{Hough-Krishnapur-Peres} \cite{berzin2022kacriceformulacontemporaryoverview}, \cites{lundberg-ramachandran-JLMS17, lerario-lundberg-PLMS16} for a detailed introduction to it and some instances of its applications. We will use the following complex variant of the Kac-Rice formula. In comparison to the more standard method of counting roots using contour integration, the following method has the advantage that one can count critical points/roots in any region $U$ by simply changing the domain of integration to $U$ in the formula. 

\begin{prop}\label{prior-kac-rice}
    Let $P$ be a polynomial of degree $m$ with only simple roots. Let $U \subset \mb{C}$ be an open set and let $\mathscr{N}(P,U)$ denote the number of roots of $P$ in $U$. Then 
     \begin{align}\label{p1}
        \mathscr{N}(P,U)  = \lim_{\epsilon \to 0} \frac{1}{\pi \epsilon^2} \int_{U} |P'(z)|^2 \mathbbm{1}_{\{|P(z)|<\epsilon\}} A(dz).
    \end{align}
    Moreover, 
    \begin{equation}\label{sup_epsilon}
        \sup_{\epsilon \in (0,\infty)} \frac{1}{\pi \epsilon^2} \int_{U} |P'(z)|^2 \mathbbm{1}_{\{|P(z)|<\epsilon\}} A(dz) \leq m .
    \end{equation}
\end{prop}

\begin{proof}
    Let $\mathscr{N}(P,U) = k$. Then, for $\epsilon$ small enough, the set $\{|P(z)| < \epsilon\} = U_1\cup U_2\cup...\cup U_k$, where $U_i\subset U$ are disjoint open connected sets each containing a single root $p_i$ of $P$. Since each root of $P$ is simple, $P'(p_i) \neq 0$. Hence, by the inverse function theorem, for $\epsilon$ small enough, $P$ is injective on each $U_i$. Clearly, $P(U_i) = B(0,\epsilon)$ and $Area(P(U_i)) = \pi \epsilon^2$.  Also, by Lemma \ref{area Lemma}-$(a)$, 
\[ Area(P(U_i)) = \int_{U_i} |P'(z)|^2 A(dz),\]
which implies 
\[ k\pi \epsilon^2 = \int_U |P'(z)|^2\mathbbm{1}_{\bigl\{|P(z)| < \epsilon\bigr\}} A(dz)  ,\]
for all $\epsilon$ small enough. In particular, dividing by $\epsilon^2$ and then letting $\epsilon \to 0$, we obtain \eqref{p1}. The estimate \eqref{sup_epsilon} follows easily from Lemma \ref{area Lemma}-$(b)$.\\
\end{proof}

\subsection{Large deviation theory in the heavy tail regime.}
The large deviation theory for random walks with heavy tail increments is qualitatively very different from the setting of Cram\'er's theorem. This has been well studied in the literature, see \cites{SVNagaevLargedeviationsofsumsofrandomvariables,cline-hsing, Linnik-Ontheprobabilityoflargedeviationsforthesumofindependentvariables, Denisov-Dieker-Shneer-Largedeviationsforrandomwalksundersubexponentiality:thebig-jumpdomain, Mikosch-Embrechts-Claudia, Mikosch-Nagaev-Largedeviationsofheavytailedsumswithapplicationsininsurance, Heyde-Onlargedeviationprobabilitiesinthecaseofattractiontoanon-normalstablelaw, Heyde-Acontributiontothetheoryoflargedeviationsforsumsofindependentrandomvariables, Heyde-Onlargedeviationproblemsforsumsofrandomvariableswhicharenotattractedtothenormallaw, Nagaev-I, Nagaev-II} and references therein. For our purpose, the paper by Cline-Hsing \cite{cline-hsing} will be the most relevant. 

\begin{thm}[Theorem $3.3$, \cite{cline-hsing}] \label{main-use-result-cline-hsing}
    Let $\{Y_k\}_{k\geq 1}$ be an i.i.d. family of random variables with common distribution $F$ and $S_n = Y_1 + Y_2 + ...+ Y_n$. Assume that $\mb{P}[|Y_1| >x]$ is regularly varying with index $-2$ and $\sup_{x \geq 0} F(-x)/(1-F(x)) <\infty$. Assume $Y_1$ has mean zero\footnote{Note that the mean of $Y_1$ is well defined since the tail of $Y_1$ is regularly varying with index $-2$.} and let $t_n$ be a sequence such that 
    
    \[ \lim_{n\to \infty} n(1-F(t_n)) = \lim_{n\to \infty} \frac{n \log t_n \mb{E}[Y_1^2 1_{|Y_1| \leq t_n }]}{t_n^2}= 0.\]

Then, \[ \lim_{n\to \infty} \sup_{s \geq t_n} \biggl|\frac{\mb{P}[S_n >s]}{n (1 - F(s))} -1 \biggr| =0.\]
\end{thm}

We will be interested in utilizing the above result for random variables $Y_k= Y_k(r) := r - Re(1/(r-X_k))$ for $r\in (0,1)$. Note that $Y_k(r)$ is a heavy tailed random variable having $p$ moments for $p<2$, see Lemma \ref{lemma: p<2 moment bound for uniform on disk}. In fact, tails of these random variables are regularly varying with index $-2$, see Lemma \ref{distribution of Y ;lemma} for a precise description of their distribution function. As observed and crucially put to use in \cite{Kabluchko-Seidel}, they lie in the non-normal domain of attraction of the normal distribution. Applying Theorem \ref{main-use-result-cline-hsing} to our setting with $W_n(r)= Y_1(r) + Y_2(r) + ...+ Y_n(r)$ and $t_n =cn$ gives 
\begin{equation}\label{intermediate-bound}
        \mb{P}[\Re(W_n(r)) > cn] \sim n \mb{P}\left[r -\Re\Bigl(\frac{1}{r-X_k}\Bigr) > cn \right] \sim C/n.
    \end{equation} for some constant $C$.

In Section \ref{section:A Large Deviated anticoncentration/small ball estimate}, we strengthen the above result to a stronger estimate which also takes into account the anticoncentration/small-ball estimate in the spirit of the Kolmogorov-Rogozin inequality. Anticoncentration inequalities have been intensively studied in the literature, see \cite{Nguyen-Van-Smallballprobabilityinversetheoremsandapplications} and references therein for an introduction to this subject. 


\section{An a priori localization of critical points outside the lemniscate}
\label{the a priori containment bound}
In view of Lemma \ref{components=critical}, we are interested in counting the number of critical points of $P_n$ outside the lemniscate $\Lambda_n$. We first obtain the following a priori information on the location of critical points outside $\Lambda_n$. We will later use this information towards proving \eqref{mainresult}.

\begin{figure}
    \centering
    \includegraphics[width = 0.475\textwidth]{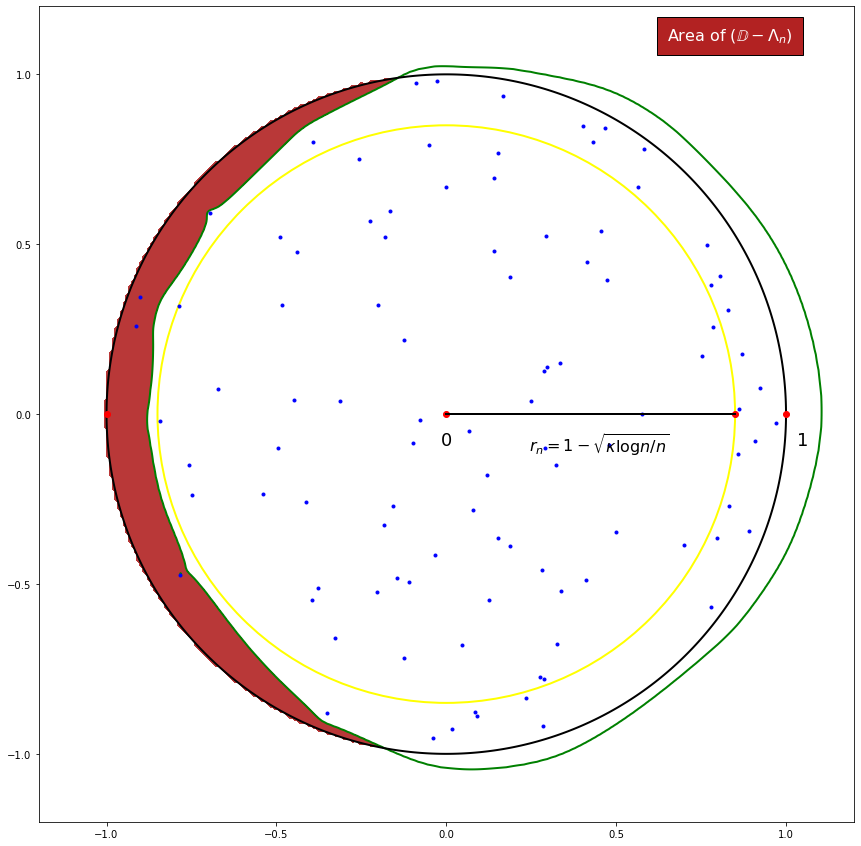}
    \includegraphics[width = 0.475\textwidth]{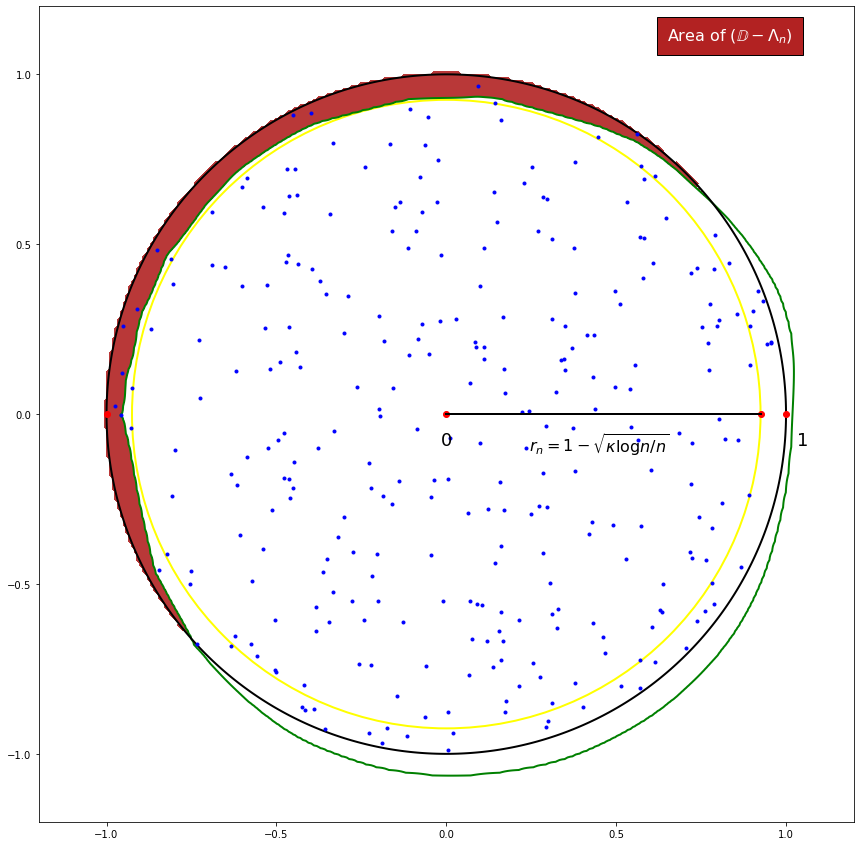}
    \caption{The figure illustrates that, for $n=100$ (left) and $n=300$ (right), 
the disk centered at the origin with radius $1-\tfrac{\kappa \log n}{n}$ (shown in yellow) 
lies entirely inside the lemniscate $\Lambda_n$ (shown in green). 
The region between the disk and the lemniscate is highlighted in red.
}\label{fig: location of lemniscates}
    \end{figure}

\begin{prop}[The inradius bound for $\Lambda_n$]
\label{location of the components of the lemniscate: lemma}
      For every $A>0$ there exists $\kappa=\kappa(A)>0$ such that, with $r_n:=\kappa\sqrt{\frac{\log n}{n}}$, we have
     \begin{align} \label{location of components of the lemniscate :eq}
         \mb{P}\left(\overline{\mb{D}}_{1-r_n} \vspace{.05in} \subset \vspace{.05in} \Lambda_n \right) \geq 1 - \frac{1}{n^A}.
     \end{align}
\end{prop}
     
\begin{proof}
The proof is an adaptation of ideas from \cite{KLR}, and it is based on concentration inequalities along with a chaining argument. By the definition of $\Lambda_n$, it is enough to show that the following holds.
    \begin{align*}
        \mb{P}\left( \sup_{z \in B(0,1-r_n)}|P_n(z)| <1 \right) \geq 1 - \frac{1}{n^A}.
    \end{align*}
    The proof proceeds as follows. First, we fix a point $z_0 \in \partial B(0,1-r_n)$ and show that 
the probability $\mathbb{P}\!\left(|P_n(z_0)| < \tfrac{1}{2}\right)$ decays polynomially in $\frac{1}{n}$. Next, by selecting $n^3$ 
uniformly distributed points on $\partial B(0,1-r_n)$, we obtain an upper bound for 
$\sup_{z \in B(0,1-r_n)} |P(z)|$ in terms of the values of $|P|$ at these points using Lemma \ref{Sup bound of polynomials}. Finally, an application of the union bound yields the desired result. 
We now begin with the probability estimates.
    \begin{align*}
                &\mathbb{P}\left[\left\{|P_n(z_0)|<\frac{1}{2}\right\}\right]=\mathbb{P}\Big(\sum_{i=1}^n\log|z_0-X_i|<-\log2\Big) \\
                &=\mathbb{P}\left(\sum_{i=1}^n(\log|z_0-X_i|-\mathbb{E}[\log|z_0-X_i|])<-\log2-n\mathbb{E}[\log|z_0-X_i|]\right) \\
                &=1-\mathbb{P}\left(\sum_{i=1}^n(\log|z_0-X_i|-\mathbb{E}[\log|z_0-X_i|])>-\log2-n\frac{(1-r_n)^2-1}{2}\right),
    \end{align*}
        where in the last line we used the fact that $\mathbb{E}[\log|z-X_i|]=\frac{|z|^2-1}{2}$ (c.f. \cite[Example 1.7]{KLR}). Now, we use Bennett's inequality with the mean zero random variables $Y_j=\log|z_0-X_j|-\mb{E}[\log|z-X_j|], b=\log2+\frac{1}{2}$ and $t=-\log2-n\frac{(1-r_n)^2-1}{2}$ to arrive at
        \begin{align}\label{probability bound for fixed $Z_0$}
            \mathbb{P}\left[\left\{|P_n(z_0)|<\frac{1}{2}\right\}\right] \geq 1- \exp\Big(-\frac{\nu}{b^2}h\left(\frac{bt}{\nu}\right)\Big).
        \end{align}
        Since all the moments of $Y_j$'s are uniformly bounded in the unit disk, $\nu=\Theta(n)$, hence
        \begin{align}\label{bt/nu going to zero}
             \frac{bt}{\nu} & \gtrsim \frac{\left(\log2+\frac{1}{2}\right) (-\log2-n\frac{(1-r_n)^2-1)}{2})}{n}
             \gtrsim \frac{n \kappa\sqrt{\frac{\log n}{n}} }{n}
             \gtrsim \kappa \sqrt{\frac{\log n}{n}} \longrightarrow 0.
        \end{align}
        In view of \eqref{bt/nu going to zero}, we approximate $h(u):=(1+u)\log(1+u)-u$ using Taylor's theorem near $0$. Computing the derivatives of $h$ at $0$, we obtain \[
            h(0)=h'(0)=0, \quad h''(0)=1, \quad h'''(u)=-\frac{1}{(1+u)^2}.
        \]
        Hence, for sufficiently small $u$,
        \begin{align*}
            h(u)&=\frac{u^2}{2}h''(0)+\frac{u^3}{6}h'''(\xi u)
            \geq \frac{u^2}{2}h''(0)+\frac{u^3}{6}h'''(0)
            \geq \frac{u^2}{4}.
        \end{align*}
        Combining these with the fact that $h$ is increasing, we obtain
        \begin{align}\label{bound for h(bt/nu)}
            \frac{\nu}{b^2}h\left(\frac{bt}{\nu}\right) \gtrsim n h\left(\kappa\sqrt{\frac{\log n}{n}}\right)  \gtrsim \kappa \log n.
        \end{align}
    Substituting \eqref{bound for h(bt/nu)} into 
\eqref{probability bound for fixed $Z_0$}, we obtain that there exists an absolute constant $C$ such that, for sufficiently large $n$,
    \begin{align}\label{final bound used for net argument}
        \mathbb{P}\left[\left\{|P_n(z_0)|<\frac{1}{2}\right\}\right] \geq 1- \exp\Big(-C\kappa\log n\Big) \geq  1-\frac{1}{n^{C\kappa}}.
    \end{align}
Let $\{\tildea{w}_k\}_{k=1}^{n^3}$ denote the $n^3$-th roots of unity, and consider their scaled version $\{w_k:=(1-r_n)w_k\}_{k=1}^{n^3}$, which we use as a net to apply Lemma~\ref{Sup bound of polynomials}. Then, applying a union bound, we obtain
\begin{align*}
    \mb{P}\left( \sup_{z \in B(0,1-r_n)}|P_n(z)| < 1 \right)
&\;\geq\; \mathbb{P}\!\left\{ \max_{1 \leq k \leq n^3} 
\big( \log |P_n(w_k)| + \log 2 \big) < 0 \right\}\\
&\;\geq\; 1- \mathbb{P}\!\left\{ \cup_{1 \leq k \leq n^3} 
\big( \log |P_n(w_k)| + \log 2 \big) > 0 \right\}\\[.5em]
&\;\geq\; 1- n^3\mathbb{P}\!\left\{ 
\big( \log |P_n(w_1)| + \log 2 \big) > 0 \right\}\\
&\;\geq\; 1- \frac{1}{n^{\kappa C-3}},
\end{align*}
where we have used the bound from \eqref{final bound used for net argument} in the last inequality.
\end{proof}

    
As a result of the above proposition, we obtain the following reduction while determining the asymptotic of $\mb{E}[\mathscr{C}(\Lambda_n)]$. In what follows, we denote the annulus $ \mb{A}_n = \{ z  \hspace{1mm}\bigl| \hspace{1mm} 1 - r_n<|z| <1\}$ and $\{\beta_j\}_{j=1}^{n-1}$ to be the set of critical points of $P_n$. Define
\begin{equation}\label{def-tilde-for-C}
\tildea{\mathscr{C}}(\Lambda_n) := 1 + \left|\left\{ j: \beta_j \in \Lambda_n^c \cap \mb{A}_n \right \}\right|.
\end{equation}

\begin{Lemma}\label{tilde-C} Let $A> 1/2$, $\kappa = \kappa(A)$ and $r_n$ be given by Proposition \ref{location of the components of the lemniscate: lemma}. Then we have 
 \[ \Big|\mb{E}[\mathscr{C}(\Lambda_n)]  -  \mb{E}[\tildea{\mathscr{C}}(\Lambda_n)] \Big| = o(\sqrt{n}).\]
\end{Lemma}

\begin{proof}
Note that 
\[ \mathscr{C}(\Lambda_n) - \tildea{\mathscr{C}}(\Lambda_n) = |j : \beta_j \notin \Lambda_n \cup \mb{A}_n|.\]
Note that the event $\{ \overline{\mb{D}}_{1-r_n} \vspace{.05in} \subset \vspace{.05in} \Lambda_n\}$ implies $\mb{D} \subset \Lambda_n \cup \mb{A}_n$. By Gauss-Lucas theorem, all critical points $\beta_j \in \mb{D}$. Hence, $|j : \beta_j \notin \Lambda_n \cup \mb{A}_n| = 0$ on the event $\{ \overline{\mb{D}}_{1-r_n} \vspace{.05in} \subset \vspace{.05in} \Lambda_n\}$. Therefore, using Proposition \ref{location of the components of the lemniscate: lemma}, we obtain that  
\[ |\mb{E}[\mathscr{C}(\Lambda_n)]  -  \mb{E}[\tildea{\mathscr{C}}(\Lambda_n)] | \leq n \mb{P}[\overline{\mb{D}}_{1-r_n} \vspace{.05in} \not\subset \vspace{.05in} \Lambda_n] \leq \frac{1}{n^{A-1}},\]
which implies the claim. 
\end{proof}

\section{Asymptotic of expected Area of \texorpdfstring{$\mb{D}\setminus\Lambda_n$}{D \ Lambda n}}
\label{the section on expected area}

In this section, we determine the asymptotic of $\mb{E}[Area(\mb{D}\setminus \Lambda_n)]$. We claim that $\mb{E}[Area(\mb{D}\setminus \Lambda_n)] \sim C/\sqrt{n}$ for some constant $C$. Note that the a priori estimate given by Proposition \ref{location of the components of the lemniscate: lemma} gives $\mb{E}[Area(\mb{D}\setminus \Lambda_n)] = O(\sqrt{\log n /n})$. 
To remove the extra $\sqrt{\log n }$ factor and enhance it to the $\sim C/\sqrt{n}$ asymptotic, one has to take into account the area of the red shaded region in Figure \ref{fig: location of lemniscates}.
The following Proposition, which is a consequence of Edgeworth expansion, gives us the required estimate along with the error bound. For the random variable $\log|r-X_1|$ and sequence $C_n$, let us write   
\begin{align}\label{some-nota}
    &u(r):= \mb{E}[\log|r-X_1|]  \\  &\sigma(r):= \sqrt{\Var\big(\log|r-X_1|\big)} \nonumber\\ & \gamma_3(r):= \mb{E}[(\log|r-X_1| -u(r))^3] \nonumber\\ & Q_{1,r}(x):=  \frac{-\gamma_3(r)}{6\sqrt{2\pi} \sigma(r)^3} (x^2 -1) e^{-x^2/2} \nonumber \\  
    &C_r(n):=\frac{nC_n- n u(r)}{\sqrt{n}\sigma(r)}.\nonumber
\end{align}
The function $u(r)$ can be explicitly computed and it turns out to be $u(r) = \frac{r^2-1}{2}$ for $r\in [0,1]$, see Example $1.7$, \cite{KLR}.

\begin{prop}\label{edgeworth-prop}
    For any sequence $C_n$ of non-negative numbers such that $\sqrt{n}C_n \to 0$, 
    \begin{align}\label{Area-1}
        \mb{P}&\bigl[ \log |P_n(X_0)| > nC_n, X_0 \in \mb{A}_n\bigr] \\& \qquad\quad= 2 \bigintsss_{ 1- \kappa\sqrt{\frac{\log n}{n}}}^1 (1- \Phi(C_r(n)))rdr + 2 \bigintsss_{ 1- \kappa\sqrt{\frac{\log n}{n}}}^1 \frac{Q_1(C_r(n))}{\sqrt{n}}rdr + O\left(\frac{1}{n} \right), \nonumber
    \end{align}
    where $\Phi$ is the distribution function of a standard normal random variable. In particular, by taking $C_n = 0$, 
\begin{equation}\label{Area-2}
 \mb{E}[Area(\mb{D}\setminus \Lambda_n)] = 2 \pi\bigintsss_{ 1- \kappa\sqrt{\frac{\log n}{n}}}^1 \biggl(1- \Phi\biggl(\frac{\sqrt{n}(1-r^2)}{2\sigma(r)}\biggr)\biggr)rdr + O\biggl(\frac{1}{n}\biggr).
\end{equation}
    Moreover, as $n\to \infty$, 
    \begin{equation}\label{gauss-integral-converge}
      \sqrt{n} \bigintsss_{ 1- \kappa\sqrt{\frac{\log n}{n}}}^1 (1- \Phi(C_r(n)))rdr \to   \sqrt{\frac{1}{2\pi}Var(\log(|1-X_1|))}.
    \end{equation}
\end{prop}

\vspace{5mm}
To prove the above proposition, we will need the following lemmas.

\begin{Lemma}\label{estimates for the characteristic fuction using bobkov et al} Let $p_r(x)$ denote the density of the random variable $\log|r-X_1|$ and let $\nu_r(t) = \mb{E}[e^{it\log|r-X_1|}]$ be its characteristic function. Then, $p_r(x)$ is uniformly bounded for $x\in \mb{R}$, and $r\in [1/2, 1]$. Moreover, for every $\delta>0$, there exists a $\eta = \eta (\delta) > 0$ such that
    \begin{align}
        \sup_{\frac{1}{2}\leq r \leq 1} \sup_{|t|> \delta} |\nu_r(t)| \leq  1-\eta.
    \end{align}
\end{Lemma}
\begin{proof}
We will use Theorem \ref{bobkov et al Theorem}. The second moment of $\log|r-X_1|$ is finite and since it is continuous in $r$, it is bounded above and below uniformly in $r$, see \cite[lemma 2.3]{Ghosh-numberofcomponents}. Next, we claim that the density of $\log|r-X_1|$ is bounded uniformly in $r$. Since $X_1$ is uniformly distributed on $\mb{D}$, it is easy to see (c.f. Figure \ref{fig: density of log|r-X_2|}-Right) 
    \begin{align}\label{density of log|r-X_2|}
        \mb{P}\left(\log|r-X_1| \leq x  \right) = \begin{cases}
                                                    & e^{2x} \hspace{1in}  \quad  \quad  \textit{for} \quad  x \leq \log(1-r)\\
                                                    & 1 \quad \hspace{1.12in} \quad  \textit{for} \quad x \geq \log2 .\quad  \quad 
                                                  \end{cases}
    \end{align}
    Let $p_r(x)$ be the density of $\log|r-X_1|$. Since $\log(1-r)<0$, from \eqref{density of log|r-X_2|} we see that $|p_r(x)| \leq2$ for all $x \in (-\infty, \log(1-r)) \cup (\log2, \infty)$. For $x \in [ \log(1-r),\log2] $, we show that the distribution function of $\log|r-X_1|$ is a Lipschitz function. Let $x \geq y$, then from a simple geometric consideration (c.f. Figure \ref{fig: density of log|r-X_2|}-Left), it follows that
    \begin{align*}
         &|\mb{P}\left(\log|r-X_1| \leq x  \right)-\mb{P}\left(\log|r-X_1| \leq y  \right)| \\
         =&|\mb{P}\left(|r-X_1| \leq e^x  \right)-\mb{P}\left(|r-X_1| \leq e^y  \right)|\leq \pi |e^{2x}-e^{2y}| \leq 4\pi |x-y|.\\
          \end{align*}
    By Radamacher's theorem, this implies that the density $p_r(x)$ is bounded by $4\pi$. Now using \eqref{bobkov eq:bound1} and \eqref{bobkov eq:bound2} from Theorem \ref{bobkov et al Theorem}, we see that 
    \begin{align*}
         \sup_{|t|> \delta} |\nu_r(t)| < \max \left\{ 1- \frac{c_116\pi^2}{\sigma(r)},1- \frac{c_2t^2}{16\pi^2} \right\} < 1- \min  \left\{ \frac{c_116\pi^2}{\sup_{\frac{1}{2}\leq r \leq 1} \sigma(r)}, \frac{c_2\delta^2}{16\pi^2} \right\}.
    \end{align*}
   Since the right-hand side of the above equation does not depend on $r$, we may choose 
\[
\eta = \min \left\{ \frac{c_1 16\pi^2}{\sup_{\tfrac{1}{2}\leq r \leq 1} \sigma(r)}, \ \frac{c_2 \delta^2}{16\pi^2} \right\},
\]
which completes the proof.
\end{proof}
\begin{figure}
    \centering
    \includegraphics[width = 0.48\textwidth]{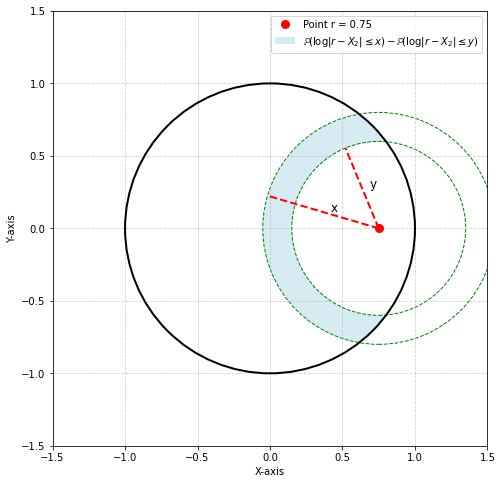}
    \includegraphics[width = 0.5\textwidth]{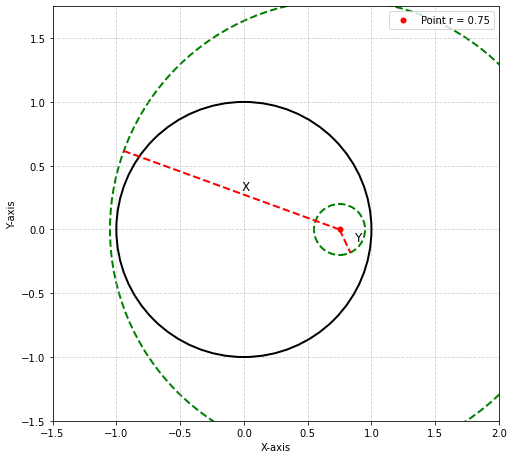}
    \caption{}\label{fig: density of log|r-X_2|}
    \end{figure}

\begin{Lemma}\label{gauss-comp}
    Let $Z$ be a normal random variable with mean $0$ and variance $1$. Then, for a sequence $C_n\geq 0$ such that $\sqrt{n}C_n \to 0$, with  $C_r(n)$ given by \eqref{some-nota}, 
    \[ \lim_{n\to \infty} \sqrt{n}\int_{1- \kappa \sqrt{\log n/n}}^1 \mb{P}[Z > C_r(n)]rdr = \frac{\sigma(1)}{\sqrt{2\pi}}. \]
\end{Lemma}

\begin{proof}
    Using the change of variable $y= y(r) = \sqrt{n}(1-r^2)/2$, we get 
    \begin{align*}
        &\sqrt{n}\bigintss_{ 1- \kappa \sqrt{\frac{\log n}{n}}}^1 \mathbb{P}\left( Z > C_r(n)\right)rdr \nonumber \\[.75em]
        & = \bigintss_{0}^{y(1- \kappa \sqrt{\log n/n})} \mathbb{P}\left( Z > (\sqrt{n}C_n + y)\bigl/\sigma(\sqrt{1- 2y/\sqrt{n}})\right)dy.\nonumber
    \end{align*}
    Note that $y(1- \kappa \sqrt{\log n/n}) \to \infty$ and by the dominated convergence theorem, the above integral converges to 
    \begin{align*}
        \int_0^\infty \mb{P}(Z > y/\sigma(1))dy = \sigma(1) \int_0^\infty \mb{P}(Z > y)dy = \frac{\sigma(1)\mb{E}[|Z|]}{2} = \frac{\sigma(1)}{\sqrt{2\pi}}. \quad \quad \quad\quad \quad \quad \qedhere
    \end{align*} 
\end{proof}

\begin{proof}[Proof of Proposition \ref{edgeworth-prop}]
    Using the rotational symmetry of the random variables $X_1, X_2, ...$ and polar coordinates, we get
\begin{align}\label{proposition for main term}
    &\mathbb{P}\left(\log|P_n(X_0)| > n C_n, X_0 \in \mb{A}_n\right) \nonumber  \\ 
    &=\frac{1}{\pi}\bigintsss_{0}^{2\pi}\bigintsss_{ 1- \kappa\sqrt{\frac{\log n}{n}}}^1 \mathbb{P}\left(\log|P_n(re^{i \theta})| > n C_n\right) rdrd\theta \nonumber \\
    &=2\bigintsss_{ 1- \kappa\sqrt{\frac{\log n}{n}}}^1 \mathbb{P}\left({\sum_{k=1}^n\log|r-X_k|} >nC_n\right) rdr.
\end{align}
Using the notation \eqref{some-nota}, the above integral transforms into 
\begin{align}\label{before applying edgeworth}
   2 \bigintsss_{ 1- \kappa\sqrt{\frac{\log n}{n}}}^1 \mathbb{P}\left(\frac{\sum_{k=1}^n(\log|r-X_k|- u(r))}{\sqrt{n}\sigma(r)} > C_r(n)\right)rdr.
\end{align}
We now apply Theorem \ref{edgeworth expansion Theorem} with $Y_k = Y_k(r) = \log|r-X_k|- u(r)$ to estimate the probability inside the integral in \eqref{before applying edgeworth}. It can be easily checked that 
$\sup_{r\in [0,1]}\mb{E}[|Y_k(r)|^4]$ is finite, see e.g. \cite[Lemma 2.3]{Ghosh-numberofcomponents}. Hence, from Theorem \ref{edgeworth expansion Theorem}, we obtain 
\begin{align}\label{main probability bound after using the edgeworth expansion}
    \left| F_{n}(x)- \Phi(x)- \frac{Q_{1}(x)}{\sqrt{n}}\right| \leq C \left( E_1(r,n)+ E_2(r,n)\right),
\end{align}
where
\begin{align*}
    &E_1(r,n)\\ &:=\frac{1}{\sigma^3 \sqrt{n}} \frac{1}{(1+|x|)^3} \int_{|y|\geq \sigma\sqrt{n}(1+|x|)} |y|^3 dF(y)+\frac{1}{\sigma^4 n}\frac{1}{(1+|x|)^4} \int_{|y|< \sigma\sqrt{n}(1+|x|)}  |y|^4 dF(y),\\[1em]
    &E_2(r,n):=\left( \sup_{|t|\geq \delta(r)} |\nu_r(t)|+\frac{1}{2n}\right)^n  \frac{n^6}{(1+|x|)^4},
\end{align*}
    where $\delta(r)=\sigma(r)^2/12\mb{E}[|Y_1|^3]$. \\
    
    We show that $E_1(r,n)$ and $E_2(r,n)$ are $O(1/n)$ uniformly in $r\in [0,1]$. For $E_1(r,n)$, 
\begin{align}\label{bound on E_1(r,n)}
    &|E_1(r,n)| \\
    & = \frac{1}{\sigma^3 \sqrt{n}} \frac{1}{(1+|x|)^3} \int_{|y|\geq \sigma\sqrt{n}(1+|x|)} \frac{|y|4}{|y|} dF(y)+\frac{1}{\sigma^4{n}} \frac{1}{(1+|x|)^4} \int_{|y|< \sigma\sqrt{n}(1+|x|)}  |y|^4 dF(y) \nonumber\\[1em]
    &\leq \frac{1}{\sigma^4 {n}} \frac{1}{(1+|x|)^4} \int_{|y|\geq \sigma\sqrt{n}(1+|x|)}|y|^4 dF(y)+\frac{1}{\sigma^4{n}}\frac{1}{(1+|x|)^4} \int_{|y|< \sigma\sqrt{n}(1+|x|)}  |y|^4 dF(y) \nonumber\\[1em]
     &\leq \frac{2}{\sigma^4 {n}(1+|x|)^4} \int_{-\infty}^{\infty}|y|^4 dF(y)  \nonumber\\[1em]
     &\leq \frac{C}{{n}},
\end{align}

    where in the last line we have used \begin{align}\label{bound on sigma}
        0< \inf_{0\leq r \leq 1} \sigma(r) \leq \sup_{0\leq r \leq 1} \sigma(r) < \infty.
    \end{align}
    which is an easy consequence of the continuity of $\sigma(r)$. \\

    Now for $E_2(r,n)$, we similarly note that $\delta := \inf_{r\in [0,1]} \delta(r) >0$. Hence, using Lemma \ref{estimates for the characteristic fuction using bobkov et al}, 
    \begin{align}\label{bound on E_2(r,n)}
        E_2(r,n) \leq n^6 \left( \sup_{|t|\geq \delta} |\nu_r(t)|+\frac{1}{2n}\right)^n \leq  n^6 \left( 1 - \eta +\frac{1}{2n}\right)^n = O\biggl(\frac{1}{n}\biggr).
    \end{align}
      
It follows from above that the expression in \eqref{before applying edgeworth} satisfies 
\begin{align}\label{afermath-edgeworth}
   2 \bigintsss_{ 1- \kappa\sqrt{\frac{\log n}{n}}}^1 &\mathbb{P}\left(\frac{\sum_{k=1}^n(\log|r-X_k|- u(r))}{\sqrt{n}\sigma(r)} > C_r(n)\right)rdr  \nonumber \\ &=  2 \bigintsss_{ 1- \kappa\sqrt{\frac{\log n}{n}}}^1 (1- \Phi(C_r(n)))rdr + 2 \bigintsss_{ 1- \kappa\sqrt{\frac{\log n}{n}}}^1 \frac{Q_1(C_r(n))}{\sqrt{n}}rdr + O\left(\frac{1}{n} \right),
\end{align}
which finishes the proof of \eqref{Area-1}. The claim \eqref{Area-2} easily follows from \eqref{Area-1}, Proposition \ref{location of the components of the lemniscate: lemma} and by noting that for $C_n=0$, 
\[ \bigintsss_{ 1- \kappa\sqrt{\frac{\log n}{n}}}^1 \frac{Q_1(C_r(n))}{\sqrt{n}}rdr \lesssim \frac{1}{n}. \]
Finally, the convergence \eqref{gauss-integral-converge} follows from Lemma \ref{gauss-comp}.
\end{proof}

\begin{remark} \label{nice-informal-remark}
    One infers from the proof above that the asymptotic $\mb{E}[Area(\mb{D}\setminus \Lambda_n)] \sim \frac{C}{\sqrt{n}}$ is an outcome of an interpolation between the large deviation regime and mean behaviour regime. More precisely, since $\mb{E}[ \log|z-X_1|] = (1-|z|^2)/2 < 0$ for each fixed $z \in \mb{D}$, the event $\{\log |P_n(z)| > 0 \}$ is a large deviation event. By Cram\'er's theorem, this event will have exponentially small probability for each fixed $z \in \mb{D}$. But, for $z\in \partial \mb{D}$, since $\mb{E}[ \log|z-X_1|] = 0$, $\{\log |P_n(z)| > 0 \}$ is a mean behaviour type event which will have $\Theta(1)$ probability. The asymptotics $\mb{E}[Area(\mb{D}\setminus \Lambda_n)] \sim C\sqrt{n}$ fall out of computation as $z$ interpolates between the interior of the unit disc and the boundary of the unit disc. As a by product it follows that only $z$ at a distance $x/\sqrt{n}, x \in (0,\infty),$ from $\partial \mb{D}$ contribute positively in the asymptotics of $\mb{E}[Area(\mb{D}\setminus \Lambda_n)]$.

    \end{remark}

\section{A Large Deviated anticoncentration/small ball estimate}\label{section:A Large Deviated anticoncentration/small ball estimate}

In this section, we prove a large deviation estimate in combination with an anticoncentration/small ball estimate. For $r\in (0,1)$, let $Y_k = Y_k(r) = r - \Re(1/(r-X_k))$ and let $W_n = W_n(r) = \sum_{k=1}^n Y_k(r)$ be the random walk generated by $Y_k(r)$. For sequences $a_n,b_n$ of positive real numbers such that $a_n \leq b_n$ and $a_n, b_n = \Theta(n)$, we want to estimate 
\begin{equation}\label{we-want-this}
    \mb{P}\left(W_n(r) \in [a_n,b_n]\right).
\end{equation} 
Note from Lemma \ref{lemma: p<2 moment bound for uniform on disk}, $W_n(r)$ is a mean-zero random walk, and the law of large numbers implies $W_n(r)/n \to 0$ almost surely as $n\to \infty$. Hence, the event $\{W_n(r) \in [a_n, b_n]\}$ is a large deviation event. Moreover, we also want to account for the fact that $W_n(r)$ is squeezed in the interval $[a_n,b_n]$, which further reduces the probability of the large deviation event \eqref{we-want-this}: this is the context of anti concentration inequalities/small ball estimates. Furthermore, we also want uniformity of our estimates with respect to the variable $r$, see Section \ref{last-analysis-of-remainder} below for the basis of these aspirations. The following proposition which could be of independent interest gives us the required estimate.

    \begin{prop}\label{dev-anti-conc}

Let $a_n,b_n = \Theta(n)$, $a_n \leq b_n$ be two sequences. Then, for any $\alpha, \beta \in (1/2,1)$ there exists a constant $C$ depending on $\alpha, \beta,$ and the sequences $a_n,b_n$ such that for all 
$r \in \left[0,  1-\frac{1}{n^{\alpha}}\right]$ 
\begin{align}\label{claim-anticoncentration}
     \mb{P}(W_n(r) \in [a_n,b_n]) \leq C \biggl( \frac{b_n-a_n}{n^2} + \frac{\log n}{n} \exp\bigl(-\frac{(\log n)^{1-\beta}}{2}\bigr)\biggr).
\end{align}
    \end{prop}
    
The proof of the above proposition is done in the following subsections, where we note down some facts about the law of random variables $Y_k(r)$ and then obtain the required upper and lower bounds. At the heart of the proof of Proposition \ref{dev-anti-conc} is the large deviation theory in the regime of heavy-tailed random walks. In particular, it uses ideas from Cline-Hsing \cite{cline-hsing} but also brings in the new idea (which is using  Bennett's inequality, c.f. Theorem    \ref{Bennett's inequality} ), which significantly simplifies the implementation of the approach in \cite{cline-hsing} to our setting.

\subsection{\texorpdfstring{The distribution of $Y_k(r)$}{The distribution of Yk(r)}}

The law of $Y_k(r)$ exhibits Pareto-type tails. In fact, the density of $Y_k(r)$ is a compact perturbation of the density of a symmetric Pareto-type law, see Remark \ref{remark-on-tail} below. More precisely,
\begin{Lemma}\label{distribution of Y ;lemma} 
For $r \in (0,1)$, the distribution function of $Y_k(r)$ satisfies  
\begin{align}\label{distribution of Y}
    F_r(t) = \mathbb{P}(Y_k(r) \leq t ) = 
    \begin{cases}
         \quad\frac{1}{4(r-t)^2},  \quad &\textit{if}  \quad  \quad  -\infty < t \leq r - \frac{1}{1+r}\\ 
         \quad1- \frac{1}{4(r-t)^2},  \quad \quad &\textit{if}  \quad \quad   r + \frac{1}{1-r} \leq t < \infty,
    \end{cases}
\end{align}\label{distribution of Y 2}
and, for $r< t < r + \frac{1}{1-r}$ 
\begin{align}\label{loose-up}
     \mathbb{P}(Y_k(r) > t ) \leq 1 \wedge \frac{1}{4(t-r)^2}.
\end{align}
\begin{figure}
    \centering
    \includegraphics[width = 0.445\textwidth]{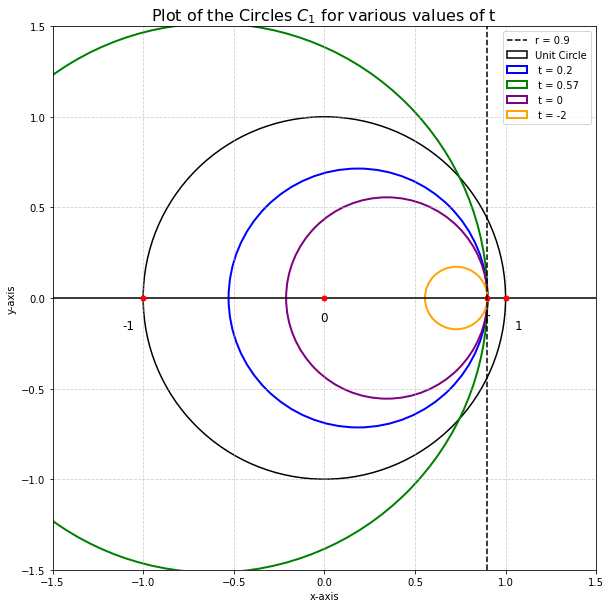}
    \includegraphics[width = 0.51\textwidth]{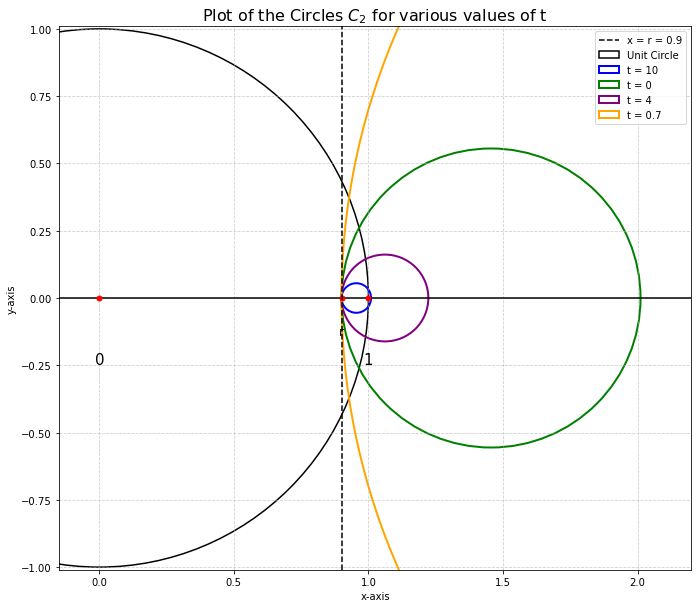}
    \caption{The unit circle is shown in black, together with curves $\mathcal{C}_t$ for various values of $t$, corresponding to the tail probabilities of $Y_r(k)$: left-sided tail (left) and right-sided tail (right).
    }\label{fig: density of Y_r(k)}
    \end{figure}
As a consequence, there exists absolute constants $C_1, C_2, >0$ such that for all $m \geq 2 $, 

\begin{equation}\label{mean-bound}
    \biggl|\int_{-m}^m x dF_r(x)\biggr| \leq \frac{C_1}{m}, 
\end{equation}
and 
\begin{equation}\label{var-bound}
   C_2 \log m \leq   \biggl|\int_{-m}^m x^2 dF_r(x)\biggr| \leq C_1 \log m.
\end{equation}
    
\end{Lemma}

\begin{proof}
    let $X_k= x+iy$, then $Y_k(r)=r-\frac{r-x}{(r-x)^2+y^2}$. For $t<r,$ a simple algebraic manipulation shows that
    \begin{align*}
        \mb{P}(Y_k(r)\leq t) =& \frac{1}{\pi} \mbox{Area}\left(r-\frac{r-x}{(r-x)^2+y^2} \leq t ; x^2+y^2 \leq 1 \right) \\
        =&  \frac{1}{\pi} \mbox{Area}\left( \Bigl(x+\frac{1}{2(r-t)}-r \Bigr)^2 +y^2 \leq \frac{1}{4(r-t)^2} ; x^2+y^2 \leq 1 \right). 
    \end{align*}
    That is, the probability is given by the normalized area of the intersection of the two disks ( c.f. Figure \ref{fig: density of Y_r(k)}).
\[
\mathcal{C}_t:=\left\{ \left(x + \frac{1}{2(r-t)} - r\right)^{2} + y^{2} \leq \frac{1}{4(r-t)^{2}} \right\} 
\quad \text{and} \quad 
\{x^{2}+y^{2} \leq 1\}.
\]  
For $t \leq r - \tfrac{1}{1+r}$, we observe that $\frac{1}{2(r-t)} \leq \frac{1+r}{2}$.  
This inequality shows that the disk 
\[
\left\{ \left(x + \frac{1}{2(r-t)} - r\right)^{2} + y^{2} \leq \frac{1}{4(r-t)^{2}} \right\} 
\]  
is completely contained inside the unit disk for $t \leq r - \tfrac{1}{1+r}$. Therefore,  
\[
\mb{P}(Y_k(r) \leq t) \;=\; \frac{1}{\pi} \, \mathrm{Area}\!\left( \left(x + \tfrac{1}{2(r-t)} - r\right)^{2} + y^{2} \leq \tfrac{1}{4(r-t)^{2}} \right) 
\;=\; \frac{1}{4(r-t)^{2}}.
\]
A similar computation yields the right tail of the distribution in \eqref{distribution of Y}. The argument above also gives \eqref{loose-up}. The estimate \eqref{mean-bound} and \eqref{var-bound} follows easily by applying integration by parts formula and using \eqref{distribution of Y} and \eqref{loose-up} (note that $Y_1(r)$ is centered, i.e. $\int_{-\infty}^{\infty} xdF_r(x) =0$ and hence $\int_{|x|\leq m }xdF_r(x)= - \int_{|x| > m }xdF_r(x)$). 
\end{proof}

\begin{remark}\label{remark-on-tail}
    We do not explicitly compute the distribution function $F_r(t)$ for the compact range $t \in [r- \frac{1}{1+r}, r + \frac{1}{1-r}]$ since it is not so neat to write an explicit expression for it. Note that in this case, as illustrated by Figure \ref{fig: density of Y_r(k)}, we have to compute the common area within two intersecting circles. However, since this is a compact interval, this does not affect the tail of $Y_k(r)$ and we do not require its explicit form for all $t\in \mb{R}$. Note, however, that the density of $Y_k(r)$ is skew and not symmetric. In fact, for $r=1$, it is one sided and supported on $\bigl(-\infty,\frac{1}{2}\bigr)$.  
\end{remark}

\subsection{Upper bound on a large deviation event}

For a sequence $c_n >0$ such that $c_n = \Theta(n)$ (we will take $c_n = a_n$ or $c_n = b_n$), we obtain an upper bound on $\mb{P} [W_n(r) > c_n]$. \\ 

For $\beta \in (1/2, 1)$, let $\epsilon = \epsilon_n = (\log n)^{-\beta}$.
Let $m= m_n = (1-\epsilon_n)c_n$. Let us write $\mu_{m}(r) = \mb{E}[Y_1(r)1_{|Y_1(r)| \leq m }]$ and $ \nu_m(r) = \mb{E}[|Y_1(r)|^21_{|Y_1(r)| \leq m }]$. From Lemma \ref{distribution of Y ;lemma}, we have $\mu_m(r)\leq C/m$ and $C_2 \log m \leq \nu_m(r) \leq C_1 \log m $. Then, 

\begin{align}\label{truncate-and-bennett}
    \mb{P} [W_n(r) > c_n] &\leq \mb{P}[W_n(r) > c_n, \max_{k=1,2,..,n}Y_k(r) \leq m ] +  \mb{P}[W_n(r) > c_n, \max_{k=1,2,..,n}Y_k(r) >m] \nonumber \\
    &\leq \mb{P}[W_n(r) > c_n, \max_{k=1,2,..,n}Y_k(r) \leq m ] + n \mb{P}[Y_1(r) >m] \nonumber \\
    &\leq \mb{P}[ \sum_{k=1}^n Y_k(r)1_{Y_k(r) \leq m} > c_n] + n  \mb{P}[Y_1(r) > m] \nonumber\\
&\leq \mb{P}[ \sum_{k=1}^n Y_k(r)1_{|Y_k(r)| \leq m} > c_n] + n  \mb{P}[Y_1(r) > m] \\    
    &\leq  \mb{P}[ \sum_{k=1}^n (Y_k(r)1_{|Y_k(r)|\leq m} - \mu_m(r)) > c_n - n\mu_m(r)] + n  \mb{P}[Y_1(r) > m] \nonumber\\
& \leq \exp\biggl( - \frac{n \nu_m(r) }{m^2 } h\biggl(\frac{m(c_n- n \mu_m(r))}{n \nu_m(r) }\biggr)\biggr) + n  \mb{P}[Y_1(r) > m], \nonumber
\end{align}
where we have used Bennett's inequality (c.f. Theorem \ref{Bennett's inequality}) to go from the fifth line to the sixth line, and the function $h$ is as in that Theorem.

\vspace{2mm}

Now, note that since $c_n = \Theta(n)$, 
\[ \frac{n \nu_m(r) }{m^2 } \times \frac{m(c_n- n \mu_m(r))}{n \nu_m(r) } = O(1),\]
\[\frac{n \nu_m(r) }{m^2 } \times \log \biggl( 1 +  \frac{m(c_n- n \mu_m(r))}{n \nu_m(r) }\biggr) = o(1),\]
and 
\[ \frac{m(c_n- n \mu_m(r))}{n \nu_m(r) } = \Theta\biggl( \frac{n}{\log n}\biggr).\]
Hence, 
\begin{align}
  &\exp\biggl( - \frac{n \nu_m(r) }{m^2 } h\biggl(\frac{m(c_n- n \mu_m(r))}{n \nu_m(r) }\biggr)\biggr) \\ & \qquad\qquad \qquad\lesssim \exp\biggl( - \frac{n \nu_m(r) }{m^2 }\times \frac{m(c_n- n \mu_m(r))}{n \nu_m(r) } \log\biggl(1+ \frac{m(c_n- n \mu_m(r))}{n \nu_m(r) }\biggr)\biggr) \nonumber \\
    & \qquad\qquad \qquad=\exp\biggl( - \frac{c_n- n \mu_m(r)}{m}\times\log\biggl(1+ \frac{m(c_n- n \mu_m(r))}{n \nu_m(r) }\biggr)\biggr) \nonumber \\
    & \qquad\qquad \qquad\lesssim \exp\biggl( - \frac{c_n}{m}\times\log\biggl(\frac{n}{\log n}\biggr) \biggr) \nonumber,
    \end{align}
    where  $\mu_m(r) = O(1/m)$ has been used in the last line. Note that 
\begin{align}
    \exp\biggl( - \frac{c_n}{m}\times\log\biggl(\frac{n}{\log n}\biggr)\biggr) &= \exp\biggl( - \frac{1}{1-\epsilon_n}\times\log\biggl(\frac{n}{\log n}\biggr)\biggr) \nonumber\\
    &= \frac{\log n}{n}\exp\biggl( - \frac{\epsilon_n}{1-\epsilon_n}\times\log\biggl(\frac{n}{\log n}\biggr)\biggr) \nonumber\\
    & \leq \frac{\log n}{n}\exp\biggl(-\frac{\epsilon_n}{2}\times\log n \biggr) \nonumber\\
    & =\frac{\log n}{n}\exp\biggl(-\frac{1}{2}\times(\log n)^{1-\beta} \biggr).\nonumber
\end{align}

Plugging the above estimates in \eqref{truncate-and-bennett}, we obtain that for some constant $C$ depending on $\alpha, \beta \in (1/2, 1)$ and the sequence $c_n$, 
\begin{align}\label{crux-upper}
    \mb{P}[W_n(r) > c_n] \leq n \mb{P} (Y_1(r) > m_n)+ C \frac{\log n}{n}\exp\biggl(-\frac{1}{2}\times(\log n)^{1-\beta} \biggr). 
\end{align}

\subsection{Lower bound on a large deviation event}
For a sequence $c_n$ as in the previous section, we now obtain a lower bound on $\mb{P}[W_n(r) > c_n]$.\\

For any $\delta \in (1/2, 1)$,  
    \begin{align*}
        \mathbb{P}\Bigl(  W_n(r) > c_n \Bigr) \geq  \mathbb{P}\Bigl(  W_n(r) > c_n ; \bigcup_{k=1,\cdots,n} \{Y_k(r) > c_n+ n^{\delta}\} \Bigr).    
    \end{align*}
    Using Bonferroni's inequality, the right hand side of the above is further lower bounded by 
    \begin{multline}\label{bonferroni}
        \Bigg[ n \mathbb{P}\Bigl( W_n(r)> c_n ; Y_1(r) > c_n+ n^{\delta} \Bigr) \\
        - \sum_{i <j} \mathbb{P}\Bigl(  W_n(r) > c_n ; Y_j(r) > c_n+ n^{\delta}; Y_i(r) > c_n+ n^{\delta} \Bigr) \Bigg].
    \end{multline}
    Now, by independence, the first term in \eqref{bonferroni} is bounded below by
    \begin{align*}
        \mathbb{P}\Bigl( W_n(r)> c_n ; Y_1(r) > c_n+ n^{\delta} \Bigr) &\geq \mathbb{P}\Bigl(  \sum_{k=2}^n  Y_k(r) > -n^{\delta} ; Y_1(r) > c_n+ n^{\delta} \Bigr)\\
        &\geq \mathbb{P}\Bigl(  \sum_{k=2}^n  Y_k(r) > -n^{\delta} \Bigr) \mathbb{P}\Bigl( Y_1(r) > c_n+ n^{\delta} \Bigr).
    \end{align*}
    Since $ Y_k(r)$ are mean-zero independent random variables, using Lemma \ref{lemma: p<2 moment bound for uniform on disk}, Lemma \ref{BE-ineq}, and Markov inequality, we get that for any $1/\delta < p<2$,   
    \begin{align*}
        \mathbb{P}\Bigl(  \sum_{k=2}^n  Y_k(r) > -n^{\delta} \Bigr) =1-  \mathbb{P}\Bigl(  \sum_{k=2}^n  Y_k(r) \leq -n^{\delta} \Bigr) \geq 1- \frac{C_p}{n^{\delta p - 1}},
    \end{align*}
    for some constant $C_p$. Also, the second term in \eqref{bonferroni} is bounded above by 
    \begin{align*}
        \mathbb{P}\Bigl(  W_n(r) > c_n ; Y_j(r) > c_n+ n^{\delta}; Y_i(r) > c_n+ n^{\delta} \Bigr) \leq \mathbb{P}\Bigl( Y_1(r) > c_n+ n^{\delta} \Bigr)^2.
    \end{align*}
    
    Plugging the above estimates into \eqref{bonferroni} and using Lemma \ref{distribution of Y ;lemma}, we obtain
    \begin{align*}
         \mathbb{P}&\Bigl(  W_n(r) > c_n \Bigr) \\&\geq  n \mathbb{P}\bigl( Y_1(r) > c_n + n^{\delta} \bigr)- \frac{C_p}{n^{\delta p -2 }}\mathbb{P}\bigl( Y_1(r) > c_n + n^{\delta} \bigr) - \frac{n(n-1)}{2}\mathbb{P}\bigl( Y_1(r) > c_n + n^{\delta} \bigr)^2\\
        & \geq  n \mathbb{P}\bigl( Y_1(r) > c_n + n^{\delta} \bigr)- \frac{C}{n^{\delta p}} - C n^2 \frac{1}{n^4}\\
        & \geq  n \mathbb{P}\bigl( Y_1(r) > c_n + n^{\delta} \bigr)- \frac{C}{n^{\delta p}}.\\
    \end{align*}
    
\subsection{Proof of  Proposition \ref{dev-anti-conc}}

Note that for $ r \leq 1 - \frac{1}{n^\alpha}$ and $m_n = (1-\epsilon_n)c_n$, we have $m_n, c_n + n^{\delta}> r + \frac{1}{1-r}$. Hence, by Lemma \ref{distribution of Y ;lemma}, 
\[ \mb{P}[Y_1(r) > m_n ] = \frac{1}{4(m_n-r)^2},\]
and \[ \mb{P}[Y_1(r) > c_n + n^{\delta} ] = \frac{1}{4( c_n + n^{\delta}-r)^2}.\]

Then, applying the upper and lower bounds obtained in the previous section with $c_n = a_n$ or $c_n = b_n$, we obtain 
\begin{align}
    \mb{P}(W_n(r) \in [a_n, b_n])
    &= \mb{P}(W_n(r) >a_n) - \mb{P}(W_n(r) >b_n)\nonumber  \\
    & \leq n \mb{P} (Y_1(r) > (1-\epsilon_n)a_n ) - n \mb{P} (Y_1(r) > b_n + n^{\delta})\nonumber \\ & \qquad\qquad\qquad +  C \frac{\log n}{n}\exp\biggl(-\frac{1}{2}\times(\log n)^{1-\beta} \biggr) + \frac{C}{n^{\delta p}}. \nonumber
\end{align} 
The above easily implies \eqref{claim-anticoncentration}, which completes the proof.




\section{A Kac-Rice Formula for counting \texorpdfstring{${\mathscr{C}}(\Lambda_n)$}{C tilde(Lambda n)}}

In this section, we obtain an explicit formula for $\mb{E}[\mathscr{C}(\Lambda_n)]$ and simplify it into a usable form.  

\subsection{The first step towards counting \texorpdfstring{${\mathscr{C}}(\Lambda_n)$}{C tilde(Lambda n)}}
Recall the notation \eqref{def-tilde-for-C}. In view of Lemma \ref{tilde-C}, it suffices to prove $\mb{E}[\tildea{\mathscr{C}}(\Lambda_n)] \sim C\sqrt{n}$ to establish Theorem \ref{main-thm}. To this end, we use the following proposition to count $\tildea{\mathscr{C}}(\Lambda_n)$.

\begin{prop}\label{Kac-Rice Formula}
 For the random polynomials $P_n$, we have almost surely 
    \begin{equation}\label{C-Formula}
        \mathscr{C}(\Lambda_n) \hspace{1mm}(\textrm{resp.} \hspace{1mm} \tildea{\mathscr{C}}(\Lambda_n) ) = 1 + \lim_{\epsilon \to 0} \frac{1}{\pi \epsilon^2} \int_{\mb{D} \hspace{1mm}(\textrm{resp.} \hspace{1mm} \mb{A}_n )} |P_n''(z)|^2 \mathbbm{1}_{\{|P_n'(z)|<\epsilon, |P_n(z)| >1\}} A(dz). 
    \end{equation} 
   Moreover, 
   \begin{equation}\label{exp-tilde-C}
       \mb{E}[\tildea{\mathscr{C}}(\Lambda_n)] = 1 + \lim_{\epsilon \to 0} \frac{1}{\pi \epsilon^2} \int_{\mb{A}_n} \mb{E}\bigl[|P_n''(z)|^2 \mathbbm{1}_{\{|P_n'(z)|<\epsilon, |P_n(z)| >1\}}\bigr] A(dz).
   \end{equation}  
\end{prop}

\begin{proof}
The formula \eqref{C-Formula} follows from \eqref{p1} by taking $P = P_n'$ with $U = \{ |P_n(z)| >1 \}$, $\{ |P_n(z)| >1 \}\cap \mb{A}_n$ respectively and invoking Lemma \ref{components=critical}. Finally, \eqref{exp-tilde-C} follows from \eqref{C-Formula} by Fubini and the dominated convergence theorem, where the domination property is verified using \eqref{sup_epsilon}. 
\end{proof}

\begin{remark}\label{a-good-remark}
    A closer look at proof of the above proposition will reveal that our approach to count the critical points via Kac-Rice formula is not well suited if critical points are to be found at the boundary of the lemniscate $\Lambda_n$ for example as in unit Bernoulli-Erd\"os lemniscate $\{ |P(z)| <1 \}$ with $P(z) = z^n-1$. The simple reason behind this is that if a critical point lies on the $\partial \Lambda_n$, then only a portion of the set $\{|P_n'(z)| <\epsilon\}$ will lie outside $\Lambda_n$. As such, area of its image under $P_n'$ will not fully compensate with $\pi \epsilon^2$. Fortunately, for random polynomials $P_n(z)$ with uniformly distributed i.i.d. roots, as proven in Lemma \ref{components=critical}, such a situation does not arise almost surely.    
\end{remark}

\subsection{A further refinement of the Kac-Rice formula}

We now explain the idea behind the formula given in Proposition \ref{Kac-Rice Formula}. To bring the formula \eqref{exp-tilde-C} into a usable form, we need to further simplify the integral and expected value appearing in \eqref{exp-tilde-C}. The expected value appearing in this formula depends on the joint distribution of $(P_n''(z), P_n'(z), P_n(z))$, which are correlated with each other. The dependence on such joint distributions also appears in the original work of Kac-Rice \cites{Rice-Mathematicalanalysisofrandomnoise, Kac-II}, and further simplification from this point onwards is specific to the model under investigation. In our particular setting, we reduce \eqref{exp-tilde-C} as follows. We play a trick of introducing a new random variable $X_0$ uniformly distributed on $\mb{D}$ and independent of the sequence $\{X_n\}_{n\geq 1}$ and writing 
\begin{align}\nonumber \label{trick}
    \mathcal{T}_n(\epsilon) := \frac{1}{\pi \epsilon^2 }\int_{\mb{A}_n} \mb{E}\bigl[|P_n''(z)|^2 &\mathbbm{1}_{\{|P_n'(z)|<\epsilon, |P_n(z)| >1\}}\bigr] A(dz) \\& = \frac{1}{\epsilon^2} \mb{E}\bigl[|P_n''((X_0))|^2 \mathbbm{1}_{\{|P_n'((X_0))|<\epsilon, |P_n((X_0))| >1, X_0 \in \mb{A}_n\}}\bigr].
\end{align}  
We now utilize the symmetry between random variables $X_0, X_1, ..., X_n$.  In the limit $\epsilon \to 0$, the event $\{|P_n'(X_0)| < \epsilon\}$ converges to $\{P_n'(X_0) =0\}$, i.e. $X_0$ is a critical point of $P_n$. The idea is that instead of solving for $X_0$ in terms of $X_1, ..., X_n$ in the equation $P_n'(X_0) =0$, solve for $X_1$ in terms of $X_0, X_2,...,X_n$. This trick allows us to perform an exact computation on the formula \eqref{trick} since the equation $P_n'(X_0) =0$ is highly non-linear to solve for $X_0$ given $X_1, X_2,..., X_n$, but given $X_0, X_2,...,X_n$, it is just a linear equation in $X_1$ which can be solved explicitly. Note that in \eqref{trick}, the factor $1/\epsilon^2$ is countered by the density of the random variable $X_0$. But, since $X_1$ has the same distribution as $X_0$, we can also counter $1/\epsilon^2$ through the density of $X_1$. \\

With $\mathcal{T}_n(\epsilon)$ given by \eqref{trick}, write $\mathcal{T}_n(0) = \lim_{\epsilon \to 0} \mathcal{T}_n(\epsilon)$. From \eqref{exp-tilde-C}, we know $ \mb{E}[\tildea{\mathscr{C}}(\Lambda_n)] = 1 + \mathcal{T}_n(0)$. Let us also introduce the polynomial
\begin{align}\label{Qpolynomial}
    \quad Q_n(z)= \frac{P_n(z)}{(z-X_1)} 
\end{align}
and random walks 
\begin{align}\label{RW-Def}
     R_n(z) := \sum_{k=2}^n \frac{1}{(z-X_k)^2}, \hspace{2mm}S_n(z) := \sum_{k=2}^n \frac{1}{z-X_k}, \hspace{2mm}\tildea{S}_n(z) := \sum_{k=3}^n \frac{1}{z-X_k}. 
\end{align}
Let the event 
\begin{align}\label{definition of O_n}
    O_n &=  \bigl\{|Q_n'(X_0)| < |Q_n(X_0)|^2, \hspace{2mm}\bigl|X_0 + \frac{Q_n(X_0)}{Q_n'(X_0)}\bigr| <1, \hspace{2mm}X_0 \in \mb{A}_n\bigr\}\\ &= \bigl\{|S_n(X_0)| < |Q_n(X_0)|, \hspace{2mm}\bigl|X_0 + \frac{1}{S_n(X_0)}\bigr| <1, \hspace{2mm}X_0 \in \mb{A}_n\bigr\}.\nonumber
\end{align}

The above outlined idea is implemented in the following proposition. 

\begin{prop} \label{key-prop} With $Q_n, R_n, S_n, O_n$ defined as above, 
\begin{equation}\label{T(0)=nice}
    \mathcal{T}_n(0) = \mathbb{E}\left[ \left| 1 + 
    \frac{R_n(X_0)}
         { \left(S_n(X_0)\right)^2 }
    \right|^2 \mathbbm{1}_{O_n}\right]. 
\end{equation} 
\end{prop}

\begin{proof}
    Let $\overline X_1$ be the solution to  
\begin{align*}
    \frac{1}{X_0-\overline{X}_1}+ \frac{1}{X_0 - X_2} + \frac{1}{X_0 - X_3}+ \cdots + \frac{1}{X_0-X_n}=0.
\end{align*}
Note that 

\begin{align}\label{aa2}
   \frac{P'(X_0)}{P(X_0)} = \frac{1}{X_0-X_1}+ \cdots + \frac{1}{X_0-X_n}, \hspace{2mm} \textrm{and} \hspace{2mm}  \frac{Q'(X_0)}{Q(X_0)} = \frac{1}{X_0-X_2}+ \cdots + \frac{1}{X_0-X_n},
\end{align}
which gives \[\overline{X}_1 = X_0 + \frac{Q(X_0)}{Q'(X_0)},\]
and 
\begin{equation}\label{X-diff}
    \frac{P'(X_0)}{P(X_0)} = \frac{1}{X_0 - X_1} - \frac{1}{X_0 - \overline X_1} = \frac{X_1 - \overline X_1}{(X_0 - X_1)(X_0 - \overline X_1)}.
\end{equation} 

\vspace{2mm}
Now, on the event $\big\{|P(X_0)| > 1 ,\hspace{.05in} |P'(X_0)|<\epsilon \big \}$, $|P'(X_0)|/|P(X_0)| < \epsilon$ and using \eqref{X-diff} gives  
\begin{align*}
   \frac{|X_1- \overline X_1|}{|X_0-X_1| \cdot |X_0- \overline X_1|} < \epsilon.
\end{align*}
Since 
\begin{align}\label{PQrelations}
    P_n'(z) = (z-X_1)Q_n'(z)+Q_n(z) \quad P_n''(z) = (z-X_1)Q_n''(z)+2Q_n'(z),
\end{align}
we have 
\begin{align}\label{aa3}
    P''(X_0) &= (X_0-X_1)Q''(X_0)+2Q'(X_0) \nonumber  \\[.75em]
    &= (X_0-\overline X_1)Q''(X_0)+2Q'(X_0) +(\overline X_1-X_1)Q''(X_0)  \nonumber \\[1em]
    & =  \underbrace{- \frac{Q(X_0)}{Q'(X_0)} Q''(X_0)+ 2Q'(X_0)}_{a} - \underbrace{\frac{(\overline X_1-X_1)}{(X_0-X_1)(X_0-\overline X_1)}}_{O(\epsilon)} \underbrace{(X_0-X_1) \frac{Q(X_0)Q''(X_0)}{Q'(X_0)}}_{b}.
\end{align}


Observe that  
\[
|Q(X_0)| \leq 2^n, \quad |Q'(X_0)| \leq n \cdot 2^n, \quad |Q''(X_0)| \leq n^2 \cdot 2^n.
\]
Also, on the event \( \{|P(X_0)| > 1\} \),
\[
|P(X_0)| = |X_0 - X_1| \cdot |Q(X_0)| > 1 \quad \Rightarrow \quad |Q(X_0)| > \frac{1}{|X_0 - X_1|} \geq \frac{1}{2},
\]
and on the event \( \{|P'(X_0)| < \epsilon \}\),
\[
\left| (X_0 - X_1) Q'(X_0) + Q(X_0) \right| < \epsilon \quad \Rightarrow \quad |Q(X_0)| - |X_0 - X_1| \cdot |Q'(X_0)| < \epsilon.
\]
This yields 
\begin{align}\label{lowerboundforQ'}
    |Q'(X_0)| > \frac{|Q(X_0)| - \epsilon}{|X_0 - X_1|} \geq \frac{1/2 - \epsilon}{2} >  \frac{1}{8}.
\end{align}
 for \( \epsilon \) small enough. Hence, from \eqref{aa3}, on the event $\big\{|P(X_0)| > 1 ,\hspace{.05in} |P'(X_0)|<\epsilon \big \}$,  
 \begin{align*}
     |P''(X_0)|^2= |a|^2 + O(\epsilon),
 \end{align*}
where the constant in $O(\epsilon)$ will depend on $n$. \\
Now, from \eqref{trick}, 
\begin{align}\nonumber 
    \mathcal{T}_n(\epsilon) =& \frac{1}{\epsilon^2} \mb{E}\bigl[|P_n''((X_0))|^2 \mathbbm{1}_{\{|P_n'((X_0))|<\epsilon, |P_n((X_0))| >1, X_0 \in \mb{A}_n\}}\bigr] \\
    =& \quad \frac{1}{\epsilon^2} \mb{E}\bigl[|a|^2 \mathbbm{1}_{\{|P_n'(X_0)|<\epsilon, |P_n(X_0)| >1, X_0 \in \mb{A}_n\}}\bigr] \label{main-1}\\& \qquad + O\bigl(\frac{1}{\epsilon}\bigr) \mb{E}\bigl[\mathbbm{1}_{\{|P_n'(X_0)|<\epsilon, |P_n(X_0)| >1, X_0 \in \mb{A}_n\}}\bigr]\label{error-1}.
\end{align} 

Note that event $\{|P'(X_0)| < \epsilon\}$ is same as 
\[\{\left| (X_0 - X_1) Q'(X_0) + Q(X_0) \right| < \epsilon\}  = \biggl\{\left| X_1 - \biggl(X_0 + \frac{Q(X_0)}{Q'(X_0)}\biggr) \right| < \frac{\epsilon}{|Q'(X_0)|}\biggr\}.\]
Hence, conditioning on $X_0, X_2, X_3,..., X_n$, 
\begin{align}
    \mb{E}^1\bigl[&\mathbbm{1}_{\{|P_n'(X_0)|<\epsilon, |P_n(X_0)| >1, X_0 \in \mb{A}_n\}}\bigr] \nonumber \\ &\leq \mathbbm{1}_{|Q'(X_0)| > 1/8} \mb{P}^1\biggl[\left |X_1 - \biggl(X_0 + \frac{Q(X_0)}{Q'(X_0)}\biggr) \right| < \frac{\epsilon}{|Q'(X_0)|}\biggr].\label{for-dct}
\end{align}
     
Since $\mb{P}^1[X_1 \in B(z, \epsilon)] \leq \epsilon^2$ for all $z\in \mb{C}, \epsilon>0$, we obtain 
\[ \mb{E}\bigl[\mathbbm{1}_{\{|P_n'(X_0)|<\epsilon, |P_n(X_0)| >1, X_0 \in \mb{A}_n\}}\bigr] = O(\epsilon ^2)\]
which implies that the term \eqref{error-1} goes to zero as $\epsilon \to 0$. 
Now, for \eqref{main-1}, again conditioning on $X_0, X_2, X_3,..., X_n$, 
\begin{align}
&\mb{E}^1\bigl[|a|^2 \mathbbm{1}_{\{|P_n'(X_0)|<\epsilon,  |P_n(X_0)| >1, X_0 \in \mb{A}_n\}}\bigr] \nonumber\\ =& |a|^2 \mathbbm{1}_{X_0 \in \mb{A}_n}  \mb{P}^1\bigl[|P_n'(X_0)|<\epsilon, |P_n(X_0)| >1\bigr] \nonumber \\ =& |a|^2 \mathbbm{1}_{X_0 \in \mb{A}_n}  \mb{P}^1\bigl[ |X_1 -X_0 - Q(X_0)/Q'(X_0)| < \epsilon/|Q'(X_0)|, |X_1 - X_0| >1/|Q(X_0)|\bigr]. \label{last-term}
\end{align}

Let us write $\mathcal{A}$ for the event $\{|X_1 -X_0 - Q(X_0)/Q'(X_0)| < \epsilon/|Q'(X_0)|, |X_1 - X_0| >1/|Q(X_0)|\}$. To compute the term $\mb{P}^1\bigl[ \mathcal{A}\bigr]$ appearing in \eqref{last-term}, we note the following conclusions made on the event $\mathcal{A}$, conditional on $X_0, X_2, X_3,..., X_n$: 

\begin{enumerate}
    \item If $|Q(X_0)|/ |Q'(X_0)| < 1/ |Q(X_0)|$ i.e. $|Q'(X_0)| > |Q(X_0)|^2$, then there is no admissible value of $X_1 - X_0$ for $\epsilon$ small enough and the event $\mathcal{A}$ is empty. This is because on event $\mathcal{A}$, for $\epsilon$ small enough, $X_1 -X_0$ is close to $|Q(X_0)|/ |Q'(X_0)|$, and $|X_1 -X_0|$ is larger than $1/|Q(X_0)|$, which implies $|Q(X_0)|/ |Q'(X_0)| > 1/|Q(X_0)|$ which is a contradiction. 
\item If $|X_0 + Q(X_0)/Q'(X_0)| > 1$, there is no admissible value of $X_1$ for $\epsilon$ small enough and the event $\mathcal{A}$ is empty. This is because on event $\mathcal{A}$, for $\epsilon$ small enough, $X_1$ is close to $X_0 + Q(X_0)/Q'(X_0)$, which implies $|X_1| >1$. Since $X_1$ takes values in the unit disc $\mb{D}$, this is a contradiction. 

\item Using the same arguments as in Lemma \ref{components=critical}-$(b)$, it follows that $ |Q'(X_0)| \neq |Q(X_0)|^2$ \hspace{1mm} and $|X_0 + Q(X_0)/Q'(X_0)| \neq 1 $ almost surely\footnote{The idea behind this claim can be easily explained in words: since $X_0, X_2, X_3,..., X_n$ are i.i.d. random variables, they will not satisfy any deterministic algebraic constraint almost surely.}.

\item If $|Q'(X_0)| < |Q(X_0)|^2$, \hspace{1mm} $|X_0 + Q(X_0)/Q'(X_0)| < 1$, for $\epsilon$ small enough, admissible values of $X_1$ contribution in the event $\mathcal{A}$ is the ball of radius $\epsilon/|Q'(X_0)|$ centered at $X_0 + Q(X_0)/Q'(X_0)$. 

\end{enumerate}

Putting together above observations, we have for $\epsilon$ small enough, 
\begin{equation}\label{for-pointwise}
\mb{P}^1\bigl[\mathcal{A}\bigr] = \frac{\epsilon^2}{|Q'(X_0)|^2} \mathbbm{1}_{|Q'(X_0)| < |Q(X_0)|^2, |X_0 + Q(X_0)/Q'(X_0)| < 1}.
\end{equation} 
Plugging the above computation in \eqref{last-term} and substituting \eqref{last-term} into \eqref{main-1}, we evaluate the limit as $\epsilon \to 0$ of the expression \eqref{main-1}. We claim that the expression \eqref{main-1} converges to 
\[ \mb{E}\biggl[ \frac{|a|^2}{|Q'(X_0)|^2} \mathbbm{1}_{\{|Q'(X_0)| < |Q(X_0)|^2, |X_0 + Q(X_0)/Q'(X_0)| < 1, X_0 \in \mb{A}_n\}}\biggr].\]
The above claim follows by writing $\mb{E}$ appearing in \eqref{main-1} via conditioning as $\mb{E}[\mb{E}^1[\cdots]]$ and using the dominated convergence theorem (DCT). The required pointwise convergence in the application of DCT is given by \eqref{for-pointwise} and the required domination is given by \eqref{for-dct}. Noting that $a/Q'(X_0) = 1 + R_n(X_0)/S_n(X_0)^2$ and $\{|Q'(X_0)| < |Q(X_0)|^2, |X_0 + Q(X_0)/Q'(X_0)| < 1, X_0 \in \mb{A}_n\} = O_n$ concludes the proof. 
\end{proof}


\subsection{Working with the Kac-Rice formula} \label{analyse-the-kac-rice}
It can be seen from the proof above that for each fixed $n$ the expected value appearing in the formula \eqref{T(0)=nice} is indeed finite. In fact, just by looking at the expression \eqref{T(0)=nice}, it is not apparent why it should be finite. In fact, the mean of the random walk $\mb{E}[|R_n|]$ is infinite (note that $\mb{E}[1/|z- X_1|^2] = + \infty$). Nevertheless, since a ratio of two random walks appears in this formula, it is natural to explore the essence of the law of large numbers/ergodic theorems (more specifically ratio ergodic theorems, see \cites{Aaronson-Jon-Anintroductiontoinfiniteergodictheory, Kamae-Keane-Asimpleproofoftheratioergodictheorem} for some results in that direction). We also note that the event $O_n$ imposes certain large deviation type constraints, which is required to control the blowing up of expected value in \eqref{T(0)=nice} even though $|R_n|$ has infinite mean. 

With $R_n(z), S_n(z), \tildea{S}_n(z)$ given by \eqref{RW-Def}, we write 
\begin{align}\label{def of S_n and S_n hat}
   R_n := R_n(X_0),  \textrm{\hspace{2mm}}  S_n := S_n(X_0)  \textrm{\hspace{2mm}and\hspace{2mm}} \tildea{S}_n := \tildea{S}_n(X_0). 
\end{align}

To analyse the formula \eqref{T(0)=nice} we note that 
\[ \frac{R_n}{S_n^2} = \sum_{k=2}^n \frac{1}{((X_0-X_k)S_n)^2},\]
which gives 
\begin{align*}
   \mathbb{E}&\biggl[ \biggl| 1 + 
    \frac{R_n(X_0)}
         { \left(S_n(X_0)\right)^2 }
    \biggr|^2 \mathbbm{1}_{O_n}\biggr] \\ &= \mathbb{E}\biggl[ \mathbbm{1}_{O_n} + \biggl| \sum_{k=2}^n\frac{1}{[(X_0-X_k)S_n]^2}\biggr|^2 \mathbbm{1}_{O_n} + 2 \Re \sum_{k=2}^n\frac{1}{[(X_0-X_k)S_n]^2} \mathbbm{1}_{O_n}
     \biggr]\\[1em]
     &=\mathbb{E}\biggl[ \mathbbm{1}_{O_n} + \sum_{k=2}^n\frac{1}{|(X_0-X_k)S_n|^4} \mathbbm{1}_{O_n}+ \sum_{2\leq i\neq j\leq n}^n\frac{1}{(X_0-X_i)^2(\overline{X_0-X_j})^2|S_n|^4} \mathbbm{1}_{O_n} \\[1em] & \qquad\qquad\qquad\qquad\qquad\qquad\qquad\qquad\qquad +2 \Re \sum_{k=2}^n\frac{1}{[(X_0-X_k)S_n]^2} \mathbbm{1}_{O_n} \biggr]
    \\[1em]
    &= O(1) + (n-1) \mathbb{E}\biggl[\frac{1}{|(X_0-X_2)S_n|^4}\mathbbm{1}_{O_n}\biggr] + O\biggl(n^2\mathbb{E}\biggl[\frac{1}{|X_0-X_2|^2|X_0-X_3|^2|S_n|^4} \mathbbm{1}_{O_n}\biggr]\biggl) \\[1em]&  \qquad\qquad\qquad\qquad\qquad\qquad\qquad\qquad\qquad +  O\biggl( n\mb{E}\biggl[\frac{1}{|(X_0-X_2)S_n|^2} \mathbbm{1}_{O_n}\biggr]\biggr).
\end{align*}

We write \[ M_n := \mathbb{E}\biggl[\frac{1}{|(X_0-X_2)S_n|^4}\mathbbm{1}_{O_n}\biggr],\]
\[ R_{n}^0 := \mb{E}\biggl[\frac{1}{|(X_0-X_2)S_n|^2} \mathbbm{1}_{O_n}\biggr], \]
and 
\[R_n^1 := \mathbb{E}\biggl[\frac{1}{|X_0-X_2|^2|X_0-X_3|^2|S_n|^4} \mathbbm{1}_{O_n}\biggr].\]

Hence, from Proposition \ref{key-prop} and the above computation, we have 
\begin{equation}\label{T-expansion}
    \mathcal{T}_n(0) = (n-1) M_n + O(1) + O(nR_n^0) + O(n^2 R_n^1).
\end{equation} 

We claim that the fourth moment term $M_n$ is the main term in the asymptotic behaviour of $\mathcal{T}_n(0)$ and the terms involving $R_n^0$ and $R_n^1$ are negligible in comparison. 

\section{Analyzing the terms \texorpdfstring{$M_n$}{Mn},
\texorpdfstring{$R_n^0$}{Rn0} and,
\texorpdfstring{$R_n^1$}{Rn1}.}
\label{some-heavy-computations}



The following proposition helps us to determine the behaviour of terms $M_n, R_n^0, R_n^1$. Recall the notation $\mb{E}_i$ from Section \ref{notations}.

\begin{prop}\label{main term being probability}
    Let $O_n$ be the event defined in \eqref{definition of O_n} and $S_n$ be the random walk defined in \eqref{def of S_n and S_n hat}. Then, 
    \begin{align}\label{E2-bound}
\mathbb{E}_2\left[ \frac{1}{|S_n(X_0-X_2)|^4} \mathbbm{1}_{O_n} \right]= \mathbb{P}_2[O_n]. 
        \end{align}
    In particular, 
    \begin{align}
\mathbb{E}\left[ \frac{1}{|S_n(X_0-X_2)|^4} \mathbbm{1}_{O_n} \right]= \mathbb{P}[O_n]. 
        \end{align}
Furthermore, for any integer $k\in \mb{Z}$, 
\begin{align}\label{2^k-bound}
\mathbb{E}_2&\left[ \frac{1}{|S_n(X_0-X_2)|^4} \mathbbm{1}_{O_n} \mathbbm{1}_{\left \{ \frac{|X_0-X_2|}{|X_0-X_3|} \in [2^k, 2^{k+1})\right\}}\right] \nonumber
           \\ &\qquad \qquad\qquad\qquad = \mathbb{P}_2\left[O_n \cap \left\{ \frac{1}{|S_n||X_0-X_3|}\in [2^k, 2^{k+1})\right\}\right].
            \end{align}
\end{prop}

\begin{proof}
     By writing $(X_0 - X_2)S_n = 1 + (X_0 - X_2)\tildea{S}_n$ and using the independence of $\tildea{S}_n$ with $X_2$, we note that  
\begin{align}\label{first-area}
         \E_2\left[\frac{1}{|(X_0-X_2)S_n|^4} \mathbbm{1}_{O_n}\right] &= \frac{1}{\pi}\int_{\mathbb{D}} \frac{1}{|1+(X_0-w)\tildea{S}_n|^4}\mathbbm{1}_{O_{2,n}}(w)dA(w),
    \end{align}
where $O_{2,n}$ is the (random) subset of $\mb{D}$ consisting of those $w$ such that the constraints of the event $O_n$ holds with freezing $X_2 =w$. More precisely, 
 \begin{align}\label{O2n-def}
        &O_{2,n} = \biggl\{ w : |w| <1, \bigr| (X_0 -w)^{-1}+ \tildea{S}_n \bigl| < |X_0-w |\prod_{k=3}^n|X_0-X_k|, \\ & \qquad\qquad\qquad\qquad \qquad\qquad\qquad\biggl| X_0+ \frac{1}{(X_0 -w)^{-1}+\tildea{S}_n} \biggr|<1, X_0 \in \mb{A}_n \biggr\}.\nonumber
    \end{align}

Define $f(w) := (X_0-w)/(1+ (X_0-w)\tildea{S}_n)$. Then $f$ is  injective and holomorphic on $O_{2, n}$ with
\[ |f'(w)|^2 = \frac{1}{|1+ (X_0-w)\tildea{S}_n|^4}.\]
Using Lemma \ref{area Lemma}, equation \eqref{first-area} can be written as 
\begin{equation}\label{area-formula-once-again}
    \eqref{first-area}= \frac{1}{\pi}\int_{O_{2,n}} |f'(w)|^2 dA(w) = \frac{1}{\pi} Area(f(O_{2,n})).
\end{equation}

Also, using \eqref{O2n-def}, 
\begin{align}\label{f(K_n)}
        &f(O_{2,n}) = \Bigg\{ f(w) :  |w| <1, \bigl| (X_0-w)^{-1}+ \tildea{S}_n \bigr| < |X_0-w|\prod_{k=3}^n|X_0-X_k|,\\&  \qquad\qquad\qquad\qquad\qquad\qquad\qquad\qquad\biggl| X_0+ \frac{1}{(X_0-w)^{-1}+\tildea{S}_n} \biggr|<1, X_0 \in \mb{A}_n\Bigg \}. \nonumber
    \end{align}
    
    Next, we simplify $f(O_{2,n})$ with a change of variable $w\mapsto f(w)$.
    A simple computation gives that $w= X_0-\Big( \frac{1}{f(w)}-\tildea{S}_n\Big)^{-1}$. Writing $ v := f(w)+X_0$, we get 
    \begin{align}\label{f(K_n) modified}
        \eqref{f(K_n)} & = \Bigg\{ v -X_0  :  \left| X_0 -  \left( \frac{1}{v-X_0}-\tildea{S}_n\right)^{-1} \right|<1, \nonumber \\
       & \hspace{1in}\left| \frac{1}{v-X_0} \right| < \left| \frac{1}{X_0-v} + \tildea{S}_n \right|^{-1} \prod_{k=3}^n|X_0-X_k|,   |v| < 1, X_0 \in \mb{A}_n\Bigg \}.
    \end{align}
    Now letting $ \widehat{S}_n:= \frac{1}{X_0-v}+\tildea{S}_n$ and $\widehat{Q}(z)=(z-v)\prod_{k=3}^n(z-X_k)$, we see \eqref{f(K_n) modified} changes to 
    \begin{align*}
        & \Bigg\{ v-X_0  :  \left| X_0 + \left( \widehat{S}_n\right)^{-1} \right|<1,\left| \frac{1}{v-X_0} \right| < \left|\widehat{S}_n \right|^{-1}\prod_{k=3}^n|X_0-X_k|,  |v| < 1, X_0 \in \mb{A}_n \Bigg \} \nonumber \\[1em]
        =& \qquad\bigg\{ v-X_0  :  \left| X_0 + \frac{\widehat{Q}(X_0)}{\widehat{Q}'(X_0)} \right|<1,\left| \frac{\widehat{Q}'(X_0)}{\widehat{Q}(X_0)} \right| < |\widehat{Q}(X_0)|,  |v| < 1, X_0 \in \mb{A}_n \bigg \}.
    \end{align*}
    Since translation by $X_0$ preserves area, it follows that
    \begin{align}\label{f(K_n) modified again}
      &\mbox{Area} ( f(O_{2,n})) = \mbox{Area} \left( \Bigg\{ v :  \left| X_0 + \frac{\widehat{Q}(X_0)}{\widehat{Q}'(X_0)} \right|<1,\left| \frac{\widehat{Q}'(X_0)}{\widehat{Q}(X_0)} \right| < |\widehat{Q}(X_0)|,  |v| < 1, X_0 \in \mb{A}_n \Bigg \} \right),
    \end{align}
where $v$ is implicitly hidden in the definition of $\widehat{Q}$. Note that we are here computing the area of the region spanned by $v$ so that $v$ satisfies the constraints appearing in the right hand side of \eqref{f(K_n) modified again} and where we have conditioned on $X_0, X_3, ..., X_n$. Hence conditional on $X_0, X_3, ..., X_n$, since $X_2$ is independent of $X_0, X_3, ..., X_n$, we can write the right hand side of \eqref{f(K_n) modified again} as 
\[ \eqref{f(K_n) modified again} = \pi \mb{E}_2\left( \Bigg\{  \left| X_0 + \frac{{Q}(X_0)}{{Q}'(X_0)} \right|<1,\left| \frac{{Q}'(X_0)}{{Q}(X_0)} \right| < |{Q}(X_0)|, X_0 \in \mb{A}_n \Bigg \} \right),\]

where we have used the fact that if we replace $v$ by $X_2$ in the expression for $\widehat{Q}(z)$ we obtain $Q(z)$. Plugging the above in \eqref{area-formula-once-again} completes the proof of \eqref{E2-bound}. The proof of \eqref{2^k-bound} uses the same arguments as above. We again use the formula \eqref{first-area}, but since in \eqref{2^k-bound} we have the additional event $\{\frac{|X_0 -X_2|}{|X_0 - X_3|} \in [2^k, 2^{k+1})\}$, we have to change the set $O_{2,n}$ in \eqref{first-area} to the set $O_{2,n} \cap \{\frac{|X_0 -w|}{|X_0 - X_3|} \in [2^k, 2^{k+1})\}$. Then, proceeding similarly as above yields the formula \eqref{2^k-bound}. 
\end{proof}

The benefit of the formula \eqref{2^k-bound} is that it also allows us to control the error term $R_n^1$ using the error term $R_n^0$.

\begin{prop} \label{R_1< R_0}Let $O_n$ be the event defined in \eqref{definition of O_n} and $S_n$ is the random walk defined in \eqref{def of S_n and S_n hat}. Then, for some absolute constant $C$, 
   \begin{align}
         R_n^1 = \mathbb{E}\biggl[\frac{1}{|X_0-X_2|^2|X_0-X_3|^2|S_n|^4} \mathbbm{1}_{O_n}\biggr] \leq C\mb{E}\biggl[\frac{1}{|(X_0-X_2)S_n|^2} \mathbbm{1}_{O_n}\biggr] = C R_{n}^0.
   \end{align}
\end{prop}

\begin{proof}
    Using \eqref{2^k-bound}, conditional on $X_0, X_3, ..., X_n$, 
 \begin{align*}
         &\mathbb{E}_2\biggl[\frac{1}{|X_0-X_2|^2|X_0-X_3|^2|S_n|^4} \mathbbm{1}_{O_n}\biggr] = \mathbb{E}_2\left[ \frac{1}{|(X_0-X_2)S_n|^4}\frac{|X_0-X_2|^2}{|X_0-X_3|^2} \mathbbm{1}_{O_n} \right]\\
         &\lesssim \sum_{k \in \mb{Z}} 2^{2k} \mathbb{E}_2\left[ \frac{1}{|(X_0-X_2)S_n|^4} \mathbbm{1}_{ \left \{\frac{|X_0-X_2|}{|X_0-X_3|} \in [2^k,2^{k+1}) \right \}}\mathbbm{1}_{O_n} \right]\\
         &\lesssim \sum_{k \in \mb{Z}}2^{2k} \mathbb{E}_2\left[ \mathbbm{1}_{ \left \{\frac{1}{|S_n||X_0-X_3|} \in [2^k,2^{k+1}) \right \}}\mathbbm{1}_{O_n} \right]\\
         &\lesssim \sum_{k \in \mb{Z}}\mathbb{E}_2\left[ \frac{1}{|(X_0-X_3)S_n|^2}\mathbbm{1}_{ \left \{\frac{1}{|S_n||X_0-X_3|} \in [2^k,2^{k+1}) \right \}}\mathbbm{1}_{O_n} \right] \\
         & = \mathbb{E}_2\left[ \frac{1}{|(X_0-X_3)S_n|^2}\mathbbm{1}_{O_n} \right].
     \end{align*}
     The proof is completed by taking the expectation with respect to $X_0, X_3, ..., X_n$. 
\end{proof}

To control the behaviour of the error term $R_n^0$ (and hence $R_n^1$ by the proposition above), we will use the following estimate:

\begin{prop}\label{key proposition for soft error term}
Let $O_n$ be the event defined in \eqref{definition of O_n} and $S_n, \tildea{S}_n$ be the random walks defined in \eqref{def of S_n and S_n hat}. Then, for some absolute constant $C$, conditional on $X_0,X_3, \cdots, X_n$, 
    \begin{align}\label{Inverse-RW-estimate}
        \mathbb{E}_2\left[\frac{1}{|(X_0-X_2)S_n|^2} \mathbbm{1}_{O_n} 
     \right] \leq C \min \left \{ 1, \frac{\log(1+2|\tildea{S}_n|)}{|\tildea{S}_n|^2}  \right \}\mathbbm{1}_{X_0 \in \mb{A}_n}.
    \end{align}
\end{prop}

\begin{proof}
If $X_0 \notin \mb{A}_n$, then $O_n$ is empty and the claim follows trivially. Hence, assume $X_0 \in \mb{A}_n$. We first write $(X_0 - X_2)S_n = 1 + (X_0 - X_2)\tildea{S}_n$. Now using the independence of $\tildea{S}_n$ with $X_2$, and conditioning on $X_0,X_3, \cdots, X_n$, we get 
\begin{align*}
         \mathbb{E}_2\left[  \frac{1}{|(X_0-X_2)S_n|^2} \mathbbm{1}_{{O}_n}
     \right]&=\frac{1}{|\tildea{S}_n|^2}\mathbb{E}_2\left[  \frac{1}{\left|(X_0-X_2)+ \frac{1}{\tildea{S}_n}\right|^2} \mathbbm{1}_{O_n}
     \right]\\
     &= \frac{1}{\pi|\tildea{S}_n|^2} \int_{O_{2,n}} \frac{1}{|\alpha-w|^2} \quad dA(w),
    \end{align*}
    where $\alpha= X_0+ \frac{1}{\tildea{S}_n} $ and $O_{2,n}$ is as given by \eqref{O2n-def}. Now, by Lemma \ref{log bound for some integral: lemma}, we can bound the above integral by
    \begin{align}\label{inf sup}
         \int_{O_{2,n}} \frac{1}{|\alpha-w|^2} \quad dA(w)&\lesssim  \log\left(\frac{\sup_{w \in O_{2,n}} |\alpha-w|}{\inf_{w \in O_{2,n}} |\alpha-w|}\right) \lesssim \log\left(\frac{1+|\alpha|}{\inf_{w \in O_{2,n}} |\alpha-w|}\right) 
    \end{align}
    
  Using that $w\in O_{2,n}$ and a simple algebraic manipulation, we obtain 
    \begin{align*}
     \left| \alpha - \frac{1}{\tildea{S}_n(1+(X_0-w)\tildea{S}_n)} \right| =\left| X_0 + \frac{1}{\tildea{S}_n} - \frac{1}{\tildea{S}_n(1+(X_0-w)\tildea{S}_n)} \right|
   = \left| X_0 + \frac{X_0-w}{1+(X_0-w)\tildea{S}_n} \right| 
    &< 1.
\end{align*}

The triangle inequality now implies 
\[ |\tildea{S}_n(1+(X_0-w)\tildea{S}_n)| > \frac{1}{1+ |\alpha|}.\]
Hence, we get
 \begin{align}\label{lb}
       |\alpha-w|= \left|\frac{1+(X_0-w)\tildea{S}_n}{\tildea{S}_n} \right| > \frac{1}{|\tildea{S}_n|^2(1+|\alpha|)}. 
    \end{align}
    
Putting \eqref{lb} and \eqref{inf sup} together, we obtain
\begin{align*}
    \mathbb{E}_2\left[  \frac{1}{|(X_0-X_2)S_n|^2} \mathbbm{1}_{O_n}
     \right] \lesssim \frac{1}{|\tildea{S}_n|^2}\log(|\tildea{S}_n|^2(1+|\alpha|)^2)  \lesssim  \frac{\log(1+2|\tildea{S}_n|)}{|\tildea{S}_n|^2}.
\end{align*}
Finally, to show $\mathbb{E}_2\left[  \frac{1}{|(X_0-X_2)S_n|^2} \mathbbm{1}_{O_n}
     \right] \lesssim 1$, we simply use the bound 
     $x^2\leq 1+ x^4$ along with                                                                                                                                                                                                                                                                                                                                                                                                                                                                                                                                                                                                                                                                                                                                                                                                                                                                                                                                                                                                                                                                                                                                                                                                                                                                                                                                                                                                                                                                                                                                                                                                                                                                                                                                                                                                                                                                                                                                                 Proposition \ref{main term being probability}.
\end{proof}


\section{Asymptotic of \texorpdfstring{$M_n$}{Mn}} \label{The-M-Analysis}

In this section, we obtain the asymptotic of the term $M_n$. We prove that
\begin{equation}\label{M-asymp}
    M_n \sim \frac{C}{\sqrt{n}}
\end{equation}
where $C = \sqrt{\frac{2}{\pi} Var(\log|1-X_1|)}$. Using Proposition \ref{main term being probability}, we know that $M_n = \mb{P}(O_n)$. Recall from \eqref{definition of O_n} that 
\begin{align}\label{main probability event}
    \mb{P}[O_n]= \mb{P}\left[ |S_n(X_0)|< |Q_n(X_0)|, \hspace{2mm}  \left| X_0+\frac{1}{S_n(X_0)} \right|<1, \hspace{2mm}X_0 \in \mb{A}_n \right].
\end{align}

Recalling the notations from \eqref{def of S_n and S_n hat}, note that conditional on $X_0$, using \eqref{mean of 1/z-X}, $S_n$ is a random walk with mean $(n-1)\overline{X}_0$. Hence, it will concentrate around $(n-1)\overline{X}_0$ with high probability. We quantify it as follows. Let $\alpha \in (3/4, 1)$ and $p \in (3/2\alpha,2)$. Then, using Lemma \ref{lemma: p<2 moment bound for uniform on disk}, Lemma \ref{BE-ineq} and Markov inequality, 

For ease of notation, we write $S_n$ for $S_n(X_0)$ (see section \ref{notations}). Then,  

\begin{align}\label{use-of-BDG}
    \mb{P}\left({|S_n-(n-1)\overline{X_0}| > n^\alpha}\right)  &\leq \frac{\mb{E}\bigr[\left|{S_n-(n-1) \overline{X_0}}\right|^p \bigl]}{n^{p\alpha}} \\ \nonumber& \lesssim  {n^{1-p\alpha} \,\mb{E}\biggl[\left|{\frac{1}{X_0-X_2}-\overline{X_0}}\right|^p} \biggr] \\\nonumber
     & \lesssim n^{1-p\alpha} \\\nonumber
     & = o\biggl(\frac{1}{\sqrt{n}}\biggr).\nonumber
\end{align}

Hence, 
\begin{align*}
    \left|\mb{P}(O_n) - \mb{P}\left(O_n \mathbbm{1}_{|S_n- (n-1) \overline{X_0}| \leq n^\alpha}\right)\right| 
    \leq \mb{P}\left({|S_n- (n-1) \overline{X_0}| > n^\alpha}\right) = o\biggl(\frac{1}{\sqrt{n}}\biggr).
\end{align*}

Next, we note that 
\[ \biggl|\mb{P}\left(O_n \mathbbm{1}_{|S_n- (n-1) \overline{X_0}| \leq n^\alpha}\right) - \mb{P}\left(O_n \mathbbm{1}_{|S_n- (n-1) \overline{X_0}| \leq n^\alpha}\mathbbm{1}_{|X_0| <  1 -\frac{4}{n}}\right) \biggr| \leq \mb{P}\bigl[ |X_0| \geq 1 - \frac{4}{n}\bigr] = o\biggl(\frac{1}{\sqrt{n}}\biggr). \]

Furthermore, on the event $\{|S_n- (n-1) \overline{X_0}| \leq n^\alpha, X_0 \in \mb{A}_n\}$ we have,
\begin{align}\label{lower bound for |S_n|}
    |S_n| \geq (n-1)|\overline{X_0}|- n^\alpha \geq \frac{n}{2}-n^\alpha \geq \frac{n}{4}
\end{align}
and 
\begin{align}\label{upper bound for |S_n|}
    |S_n| \leq (n-1)|\overline{X_0}| +n^\alpha \leq 2n.
\end{align}
Hence, the event $\{|S_n- (n-1) \overline{X_0}| \leq n^\alpha, X_0 \in \mb{A}_n, |X_0| < 1 - \frac{4}{n}\}$ implies the event $\{|X_0 + \frac{1}{S_n}| <1\}$. 
Therefore, we conclude that 
\[ \mb{P}(O_n) = \mathbb{P}\left(|S_n|< |Q_n|, 1- \kappa\sqrt{\frac{\log n}{n}} \leq |X_0| < 1 - \frac{4}{n} , {|S_n- (n-1) \overline{X_0}| \leq n^\alpha}\right) + o\biggl(\frac{1}{\sqrt{n}}\biggr). \]

Also, from \eqref{upper bound for |S_n|} and \eqref{lower bound for |S_n|}, we note that the probability on the right-hand side in the above equation can be sandwiched between 
\[ \mathbb{P}\left( cn< |Q_n|, 1- \kappa\sqrt{\frac{\log n}{n}} \leq |X_0| < 1 - \frac{4}{n} , {|S_n- (n-1) \overline{X_0}| \leq n^\alpha}\right)\]
where $c$ is either $2$ or $1/4$. In the same manner as above, we can show that
\begin{align}
    &\mathbb{P}\left( cn< |Q_n|, 1- \kappa\sqrt{\frac{\log n}{n}} \leq |X_0| < 1 - \frac{4}{n} , {|S_n- (n-1) \overline{X_0}| \leq n^\alpha}\right) \\
    & \qquad\qquad \qquad\qquad= \mb{P}(cn < |Q_n|, X_0 \in \mb{A}_n) + o\biggl(\frac{1}{\sqrt{n}}\biggr).\nonumber
\end{align}

Finally, it follows from Proposition \ref{edgeworth-prop} with the choice of $C_n = \log(cn)/n$, $c\in\{ 2, 1/4\}$, that 
\[ \mb{P}(cn < |Q_n|, X_0 \in \mb{A}_n)\sim \frac{C}{\sqrt{n}}\]
with $C = \sqrt{\frac{2}{\pi} Var(\log|1-X_1|)}$. This completes the proof of the claim \eqref{M-asymp}.

\section{Estimating the error terms \texorpdfstring{$R_n^0$,$R_n^1$}{Rn0, Rn1}.} \label{last-analysis-of-remainder}
We now prove that error terms $R_n^0$ and $R_n^1$ do not contribute to the asymptotic of $\mathcal{T}_n(0)$ given by \eqref{T-expansion}.
We prove that $R_n^0 = o(1/\sqrt{n})$ and $R_n^1 = o(1/n^{3/2})$. 
This together with \eqref{T-expansion} and \eqref{M-asymp} implies that $\mathcal{T}_n(0) \sim C\sqrt{n}$ with $C = \sqrt{\frac{2}{\pi} Var(\log|1-X_1|)}$. Also, in view of Proposition \ref{R_1< R_0}, it suffices to prove 
\begin{equation}\label{strong-claim}
    R_n^0 = o(1/n^{3/2}).
\end{equation}

To this end, we use Proposition \ref{key proposition for soft error term}. Using \eqref{Inverse-RW-estimate}, we have  
\begin{align} \label{key-bound-on-R_n}
        R_n^0 \lesssim  \mb{E}\biggl[\min \left \{ 1, \frac{\log(1+2|\tildea{S}_n)|}{|\tildea{S}_n|^2}  \right \}\mathbbm{1}_{X_0 \in \mb{A}_n}\biggr].
    \end{align}
Note that conditional on $X_0$, the random walk $\tildea{S}_n \sim n \overline{X}_0$ almost surely. This heuristically suggests that the right-hand side above is $O(\log(n)/n^2)$, which is clearly $o(1/n^{3/2})$. But, since we do not have any lower bound on $\tildea{S}_n$, the almost sure convergence $\tildea{S}_n \sim n \overline{X_0}$ does not suffice to invoke the dominated convergence theorem in order to obtain a $O(\log(n)/n^2)$ bound on the expected value appearing in \eqref{key-bound-on-R_n}. We are, in fact, not able to prove an $O(\log(n)/n^2)$ bound on $R_n^0$. The final estimate that we obtain on $R_n^0$ is only marginally better than $1/n^{3/2}$, see \eqref{marginally-better} below. However, the softer bound $R_n^0 = o(1/\sqrt{n})$ can be obtained easily as follows. For $\delta > 0 $, we use \eqref{key-bound-on-R_n} to obtain 

\begin{align} \label{key-bound-on-R_n-second}
        R_n^0 &\lesssim \mb{E}\biggl[\min \left \{ 1, \frac{\log(1+2|\tildea{S}_n|)}{|\tildea{S}_n|^2}  \right \}\mathbbm{1}_{X_0 \in \mb{A}_n}\biggr] \\ 
        &\lesssim \mb{E}\biggl[ \frac{\log(1+2|\tildea{S}_n|)}{|\tildea{S}_n|^2}  \mathbbm{1}_{X_0 \in \mb{A}_n} \mathbbm{1}_{|\tildea{S}_n| \geq  \delta n}\biggr] + \mb{E}\biggl[ \mathbbm{1}_{X_0 \in \mb{A}_n}\mathbbm{1}_{|\tildea{S}_n| < \delta n}\biggr]\nonumber\\
        & \lesssim \mb{E}\biggl[ \frac{1}{|\tildea{S}_n|^{3/2}}   \mathbbm{1}_{|\tildea{S}_n| \geq  \delta n}\biggr] + \mb{P}\bigl[ |\tildea{S}_n| < \delta n, X_0 \in \mb{A}_n\bigr]\nonumber\\ 
        &\lesssim \frac{1}{n^{3/2}} + \mb{P}\bigl[ |\tildea{S}_n| < \delta n, |X_0|  \geq \frac{1}{2}\bigr], \nonumber
    \end{align}
    
    where we have used $\log(1+ x) \lesssim \sqrt{x}$ for all $x>0$.
    Also note that for $\delta$ small enough, conditional on $|X_0| \geq 1/2 $, the event $ \{|\tildea{S}_n| < \delta n \}$ is a large deviation event since it implies $\{|\tildea{S}_n - (n-2)\overline{X}_0| > (\frac{1}{2} - \delta) n \}$. Then, similarly as in \eqref{use-of-BDG}, it follows that $\mb{P}[ |\tildea{S}_n - (n-2)\overline{X}_0| > (\frac{1}{2} - \delta) n \hspace{1mm} \bigl|\hspace{1mm}X_0] = O(1/n^{1-\epsilon})$ for any $\epsilon >0$. Hence, using \eqref{key-bound-on-R_n-second}, we conclude that $R_n^0 = o(1/\sqrt{n})$. \\

    The above argument is not sufficient to prove the stronger estimate \eqref{strong-claim}. We need a better estimates on $\mb{P}\bigl[ |\tildea{S}_n| < \delta n, |X_0|  \geq \frac{1}{2}\bigr]$ than used in the previous paragraph: the weaker $O(1/n^{1-\epsilon})$ bound is not useful for \eqref{strong-claim}. To this end, we need better probability estimates on the large deviation event $\mb{P}[ |\tildea{S}_n - (n-2)\overline{X}_0| > (\frac{1}{2} - \delta) n \hspace{1mm} \bigl|\hspace{1mm}X_0]$. Using the rotational symmetry of random variables $X_k$, it is enough to look at $\Re(\tildea{S}_n(r))$, i.e. $\Re(\tildea{S}_n)$ conditional on $X_0 = r$. The $\Re(\tildea{S}_n(r))$ is a random walk with increments $\Re(1/(r-X_k))$. Then, we fall in the setting of Section \ref{section:A Large Deviated anticoncentration/small ball estimate}. By using Theorem \ref{main-use-result-cline-hsing}, one can obtain $\mb{P}[ |\tildea{S}_n - (n-2)\overline{X}_0| > (\frac{1}{2} - \delta) n \hspace{1mm} \bigl|\hspace{1mm}X_0] = O(1/n)$, see \eqref{intermediate-bound}, which is an improvement over above mentioned $O(1/n^{1-\epsilon})$ bound. However, even this $O(1/n)$ bound is also not sufficient to give us \eqref{strong-claim}. In fact, while considering $\mb{P}\bigl[ |\tildea{S}_n| < \delta n, |X_0|  \geq \frac{1}{2}\bigr]$, we not only have to account for the large deviation behaviour of $\Re(\tildea{S}_n(r))$ but also the fact that $\Re(\tildea{S}_n(r))$ is squeezed in an interval of relatively small size. This observation gives us the final leverage which allows us to prove \eqref{strong-claim}. Our requirement is precisely formulated in Proposition \ref{dev-anti-conc}. \\
    
    Relying on Proposition \ref{dev-anti-conc}, we can now prove \eqref{strong-claim}. For $\epsilon>0$ small, we split $R_n^0$ as 
    \begin{multline}\label{broken at n delta}
     R_n^0= \mathbb{E}\left[ \frac{1}{|S_n(X_0-X_2)|^2} \mathbbm{1}_{O_n} \mathbbm{1}_{\left\{|\tildea{S}_n| \geq n^{\frac{3}{4}+ \epsilon} \right\} }\right]\\ + \mathbb{E}\left[ \frac{1}{|S_n(X_0-X_2)|^2} \mathbbm{1}_{O_n }\mathbbm{1}_{\left\{|\tildea{S}_n| <n^{\frac{3}{4}+ \epsilon}\right \}}\right].
\end{multline}

For the first term in the right-hand side above, we use Proposition \ref{key proposition for soft error term} to obtain,
\begin{align*}
    \mathbb{E}\left[ \frac{1}{|S_n(X_0-X_2)|^2} \mathbbm{1}_{O_n} \mathbbm{1}_{\left\{|\tildea{S}_n| \geq n^{\frac{3}{4}+ \epsilon} \right\} }\right] \lesssim \mathbb{E}\left[ \frac{\log(1+2|\tildea{S}_n|)}{|\tildea{S}_n|^2} \mathbbm{1}_{\left\{|\tildea{S}_n| \geq n^{\frac{3}{4}+ \epsilon} \right\} }\right].
\end{align*}
Choosing $p>0$ small enough such that $\big(\frac{3}{4}+ \epsilon\big)(2-p) > \frac{3}{2}$ and using the fact that $\log(1+x) \lesssim x^p$ for all $x>0$, we obtain
\begin{align*}
    \mathbb{E}\left[ \frac{1}{|S_n(X_0-X_2)|^2} \mathbbm{1}_{O_n} \mathbbm{1}_{\left\{|\tildea{S}_n| \geq n^{\frac{3}{4}+ \epsilon} \right\} }\right] &\lesssim \mathbb{E}\left[ \frac{1}{|\tildea{S}_n|^{2-p}} \mathbbm{1}_{\left\{|\tildea{S}_n| \geq n^{\frac{3}{4}+ \epsilon} \right\} }\right]\\
    &\lesssim \frac{1}{n^{(2-p)(\frac{3}{4} + \epsilon)}} = o\left(\frac{1}{n^{\frac{3}{2}}}\right).
\end{align*}

For the second term on the right-hand side of \eqref{broken at n delta}, using Proposition \ref{key proposition for soft error term} again, 
\begin{equation}\label{first-term-of-split}
     \mathbb{E}\left[ \frac{1}{[S_n(X_0-X_2)]^2} \mathbbm{1}_{O_n } \mathbbm{1}_{\left\{|\tildea{S}_n| < n^{\frac{3}{4}+ \epsilon} \right\} }\right] \lesssim \mathbb{P} \left(X_0 \in \mb{A}_n,| \tildea{S}_n| < n^{\frac{3}{4}+ \epsilon} \right). 
\end{equation}

Using the rotational invariance of random variables $X_k$, for any fixed $\alpha> \frac{1}{2}$,
\begin{align}\label{follow-up-on-first-term-of-split}
 \mathbb{P} &\left(X_0 \in \mb{A}_n,| \tildea{S}_n| < n^{\frac{3}{4}+ \epsilon}\right) \\ &= 2\int_{1 - \kappa \sqrt{\frac{\ logn}{n}}}^1\mathbb{P} \left(| \tildea{S}_n(r)| < n^{\frac{3}{4}+ \epsilon}\right) r dr\nonumber \\ & = 2\int_{1 - \kappa \sqrt{\frac{\ logn}{n}}}^{1 -\frac{1}{n^\alpha}}\mathbb{P} \left(| \tildea{S}_n(r)| < n^{\frac{3}{4}+ \epsilon} \right) r dr + 2\int_{1 -\frac{1}{n^\alpha}}^1\mathbb{P} \left(| \tildea{S}_n(r)| < n^{\frac{3}{4}+ \epsilon} \right) r dr. \nonumber\\ \nonumber
      \end{align} 
 For $r\in [1 - \frac{1}{n^\alpha}, 1]$, we use the weaker bound $\mathbb{P} (\{| \tildea{S}_n(r)| < n^{\frac{3}{4}+ \epsilon} \}) = O(1/n^{1-\epsilon})$ as proved in \eqref{use-of-BDG} to obtain 
 \[  \int_{1 -\frac{1}{n^\alpha}}^1\mathbb{P} \left(\left\{| \tildea{S}_n(r)| < n^{\frac{3}{4}+ \epsilon} \right \}\right) r dr = O\biggl(\frac{1}{n^{1-\epsilon}} \times \frac{1}{n^{\alpha}}\biggr) = o\biggl(\frac{1}{n^{{3/2}}}\biggr)\]
 for $\epsilon$ small enough. \\
For $r\in [1 - \kappa \sqrt{\frac{\log n}{n}}, 1- \frac{1}{n^{\alpha}}]$, using the notation $W_n(r)$ from Section \ref{section:A Large Deviated anticoncentration/small ball estimate} and using Proposition \ref{dev-anti-conc},
\begin{align}\label{we-get-only-log-bound}
    \mathbb{P} \left(\left\{| \tildea{S}_n(r)| < n^{\frac{3}{4}+ \epsilon} \right \}\right) & \leq  \mathbb{P} \left(\left\{| Re(\tildea{S}_n(r))| < n^{\frac{3}{4}+ \epsilon} \right \}\right) \\ & =  \mathbb{P} \left(\left\{ W_n(r) \in [nr - n^{\frac{3}{4} + \epsilon}, nr + n^{\frac{3}{4} + \epsilon}]\right \}\right) \nonumber\\ & \lesssim \frac{1}{n (\log n)^\gamma},\nonumber
\end{align}
where $\gamma$ is any arbitrarily large constant. This implies that 
\begin{equation} \label{marginally-better}
\int_{1 - \kappa \sqrt{\frac{\ logn}{n}}}^{1 -\frac{1}{n^\alpha}}\mathbb{P} \left(\left\{| \tildea{S}_n(r)| < n^{\frac{3}{4}+ \epsilon} \right \}\right) r dr = O\biggl( \frac{1}{n (\log n)^\gamma} \times\sqrt{\frac{\log n }{ n}}\biggr)  = o\biggl(\frac{1}{n^{3/2}}\biggr).
\end{equation}
Plugging the above estimates in \eqref{follow-up-on-first-term-of-split}, \eqref{first-term-of-split}, and \eqref{broken at n delta} completes the proof of \eqref{strong-claim}.

\begin{remark}
    The error bound in the asymptotic of the main term $M_n \sim C/\sqrt{n}$ is in fact $O(1/n)$. This can be seen from Proposition \ref{edgeworth-prop}. However, the error terms $R_n^0, R_n^1$ in view of the estimate \eqref{strong-claim} are more delicate. The fact that we only get $O(1/(\log n)^\gamma)$ estimate for the error term in \eqref{mainresult} is stemming from \eqref{we-get-only-log-bound}. 
\end{remark}
\vspace{1mm}

\section{Computation of the Variance and proof of Theorem \ref{main-thm}} \label{some-comp-last}

\begin{Lemma}\label{computation for exact asymptotics}
Let $X$ be a uniform random variable in the unit disk. Then 
    \begin{equation*}
I := \mathbb{E}\left[ [\log|1-X|]^2\right] = \frac{1}{\pi} \int_0^1 \int_0^{2\pi} \left( \log \left| 1 - r e^{i\theta} \right| \right)^2 \, r \, d\theta \, dr = \frac{\zeta(2) - 1}{2}= \frac{\pi^2 - 6}{12}.
\end{equation*}
\end{Lemma}

\begin{proof}
    We begin by considering the Taylor expansion of the complex logarithm:
\begin{equation*}
\log(1 - z) = -\sum_{n=1}^\infty \frac{z^n}{n}, \quad |z| < 1.
\end{equation*}
Since $z\in \D$, substituting into polar co ordinates \( z = r e^{i\theta} \), where \( 0 \leq r < 1 \) we get
\begin{equation*}
\log(1 - r e^{i\theta}) = -\sum_{n=1}^\infty \frac{r^n e^{i n\theta}}{n}.
\end{equation*}
Taking the real part of the equation above gives
\begin{equation}\label{fourierexpansion}
\log|1 - r e^{i\theta}| = \operatorname{Re}\left[ \log(1 - r e^{i\theta}) \right] = -\sum_{n=1}^\infty \frac{r^n}{n} \cos(n\theta).
\end{equation}
Now we compute the squared integral using the expansion in \eqref{fourierexpansion}
\begin{equation}
\int_0^{2\pi} \left( \log|1 - r e^{i\theta}| \right)^2 d\theta = \int_0^{2\pi} \left( \sum_{n=1}^\infty \frac{r^n}{n} \cos(n\theta) \right)^2 d\theta.
\end{equation}
Expanding the square and switching the sum and the integral, we get:
\begin{equation*}
\sum_{m=1}^\infty \sum_{n=1}^\infty \frac{r^{m+n}}{mn} \int_0^{2\pi} \cos(m\theta) \cos(n\theta) d\theta
\end{equation*}
Using the orthogonality of cosine functions,
we see that \eqref{fourierexpansion} transforms to
\begin{equation}\label{Li2representation}
\int_0^{2\pi} \left( \log|1 - r e^{i\theta}| \right)^2 d\theta = \pi \sum_{n=1}^\infty \left( \frac{r^n}{n} \right)^2 = \pi \sum_{n=1}^\infty \frac{r^{2n}}{n^2}.
\end{equation}
This sum is the dilogarithm function, i.e.,
\begin{equation*}
\int_0^{2\pi} \left( \log|1 - r e^{i\theta}| \right)^2 d\theta = \pi \, \text{Li}_2(r^2)
\end{equation*}
Plugging this in $I$ and using the change variable: \( u = r^2 \Rightarrow r \, dr = \frac{1}{2} du \)
\begin{align*}
I &= \frac{1}{\pi} \int_0^1 \left( \pi \, \text{Li}_2(r^2) \right) r \, dr = \int_0^1 \text{Li}_2(r^2) \, r \, dr =\frac{1}{2} \int_0^1 \text{Li}_2(u) \, du
\end{align*}
Since $u \in [0,1]$, the sum in \eqref{Li2representation} converges uniformly, and we can use MCT to switch the integral and the sum to get
\begin{align*}
\int_0^1 \text{Li}_2(u) \, du &= \int_0^1 \sum_{n=1}^\infty \frac{u^n}{n^2} du = \sum_{n=1}^\infty \frac{1}{n^2} \int_0^1 u^n du = \sum_{n=1}^\infty \frac{1}{n^2(n+1)}
\end{align*}
 We use partial fractions to compute the sum.
\begin{align*}
\sum_{n=1}^\infty \frac{1}{n^2(n+1)} = \sum_{n=1}^\infty \left( \frac{1}{n^2} - \frac{1}{n} + \frac{1}{n+1} \right)  = \zeta(2) - 1
\end{align*}
Therefore, putting everything together, we get the desired equality.
\begin{equation*}
I = \frac{1}{\pi} \int_0^1 \int_0^{2\pi} \left( \log |1 - r e^{i\theta}| \right)^2 \, r \, d\theta \, dr = \frac{1}{2} \left(\zeta(2) - 1 \right) \qedhere
\end{equation*}

\end{proof}


\begin{proof}[Proof of Theorem \ref{main-thm}]We now assemble the results from the previous sections to complete the proof of the main theorem. 
By Lemma~\ref{tilde-C}, we may replace the component count by its modified version:
\begin{align*}
    \frac{1}{\sqrt{n}}\mathbb{E}[\mathscr{C}(\Lambda_n)]
    =
    \frac{1}{\sqrt{n}}\mathbb{E}[\widetilde{\mathscr{C}}(\Lambda_n)]
    + o(1),
\end{align*}
where $\widetilde{\mathscr{C}}(\Lambda_n)$ is defined in \eqref{def-tilde-for-C}. Now, combining Proposition~\ref{Kac-Rice Formula} (Kac--Rice formula) with Proposition~\ref{key-prop}, we obtain
\begin{align*}
    \frac{1}{\sqrt{n}}\mathbb{E}[\mathscr{C}(\Lambda_n)]
    =
    \frac{1}{\sqrt{n}} \lim_{\epsilon \to 0} \mathcal{T}_n(\epsilon) + o(1)
    =
    \frac{1}{\sqrt{n}}\mathbb{E}\left[
        \left| 1 + \frac{R_n(X_0)}{(S_n(X_0))^2} \right|^2
        \mathbbm{1}_{O_n}
    \right]
    + o(1),
\end{align*}
with $R_n$, $S_n$ and $\mathcal{T}_n(\epsilon)$ are as in \eqref{def of S_n and S_n hat} and \eqref{trick} respectively. Expanding the square and invoking the decomposition from \eqref{T-expansion}, we arrive at
\begin{align}\label{the last equation}
    \frac{1}{\sqrt{n}}\mathbb{E}[\mathscr{C}(\Lambda_n)]
    =
    \sqrt{n}\,M_n
    +
    \sqrt{n}\,O(R_n^0)
    +
    n\sqrt{n}\,O(R_n^1)
    +
    o(1).
\end{align}
Next, we will show $M_n$ has the principal contribution in component count and $R_n^0, R_n^1$ are negligible compared to $M_n$. The asymptotics of $M_n$ were established in Section~\ref{The-M-Analysis}, while the terms $R_n^0$ and $ R_n^1$ were handled in Section~\ref{last-analysis-of-remainder}. Combining the estimates \eqref{strong-claim}, \eqref{T-expansion} with  Proposition~\ref{R_1< R_0}, and passing to the limit in \eqref{the last equation}, we conclude that
\begin{align*}
    \lim_{n \to \infty}
    \frac{1}{\sqrt{n}}\mathbb{E}[\mathscr{C}(\Lambda_n)]
    =
    \sqrt{\frac{2}{\pi}\,\mathrm{Var}\bigl(\log|1 - X_1|\bigr)}
    =
    \sqrt{\frac{\zeta(2) - 1}{\pi}}.
\end{align*}
Finally, in the last equality above, the explicit numerical value of $\mathrm{Var}\bigl(\log|1 - X_1|\bigr)$ is computed in Lemma~\ref{computation for exact asymptotics}.
\end{proof}

\textbf{Acknowledgment}. The first author (S.G.) was partially supported by the ANRF Core Research Grant CRG/2023/002667 during his postdoctoral stay at the Indian Statistical Institute Bangalore Campus.

\bibliographystyle{siam}
\bibliography{ref}

@article {Nazarov-Sodin,
    AUTHOR = {Nazarov, Fedor and Sodin, Mikhail},
     TITLE = {On the number of nodal domains of random spherical harmonics},
   JOURNAL = {Amer. J. Math.},
  FJOURNAL = {American Journal of Mathematics},
    VOLUME = {131},
      YEAR = {2009},
    NUMBER = {5},
     PAGES = {1337--1357},
      ISSN = {0002-9327,1080-6377},
   MRCLASS = {60F10 (33C55 43A90 60D05)},
  MRNUMBER = {2555843},
       DOI = {10.1353/ajm.0.0070},
       URL = {https://doi.org/10.1353/ajm.0.0070},
}

@article {GhoshRamachandran,
    AUTHOR = {Ghosh, Subhajit and Ramachandran, Koushik},
     TITLE = {Number of components of polynomial lemniscates: a problem of
              {E}rd\"os, {H}erzog, and {P}iranian},
   JOURNAL = {J. Math. Anal. Appl.},
  FJOURNAL = {Journal of Mathematical Analysis and Applications},
    VOLUME = {540},
      YEAR = {2024},
    NUMBER = {1},
     PAGES = {Paper No. 128571, 21},
      ISSN = {0022-247X,1096-0813},
   MRCLASS = {30C85 (31A15)},
  MRNUMBER = {4758329},
MRREVIEWER = {Malik\ Younsi},
       DOI = {10.1016/j.jmaa.2024.128571},
       URL = {https://doi.org/10.1016/j.jmaa.2024.128571},
}

@article {Kabluchko-Wigman,
    AUTHOR = {Kabluchko, Zakhar and Wigman, Igor},
     TITLE = {Asymptotics for the expected number of nodal components for
              random lemniscates},
   JOURNAL = {Int. Math. Res. Not. IMRN},
  FJOURNAL = {International Mathematics Research Notices. IMRN},
      YEAR = {2022},
    NUMBER = {3},
     PAGES = {2337--2375},
      ISSN = {1073-7928,1687-0247},
   MRCLASS = {60G60 (30D99 30F99 37F99 60G15)},
  MRNUMBER = {4373237},
       DOI = {10.1093/imrn/rnaa146},
       URL = {https://doi.org/10.1093/imrn/rnaa146},
}

@article {Michelen-Xuan-Truong-Almostsurebehaviorofthezerosofiteratedderivativesof,
    AUTHOR = {Michelen, Marcus and Vu, Xuan-Truong},
     TITLE = {Almost sure behavior of the zeros of iterated derivatives of
              random polynomials},
   JOURNAL = {Electron. Commun. Probab.},
  FJOURNAL = {Electronic Communications in Probability},
    VOLUME = {29},
      YEAR = {2024},
     PAGES = {Paper No. 27, 10},
      ISSN = {1083-589X},
   MRCLASS = {60B99 (30C15)},
  MRNUMBER = {4762159},
       DOI = {10.1214/24-ECP596},
       URL = {https://doi.org/10.1214/24-ECP596},
}

@article {Byun-Sung-Soo-Lee-Reddy-their-higher-derivatives,
    AUTHOR = {Byun, Sung-Soo and Lee, Jaehun and Reddy, Tulasi Ram},
     TITLE = {Zeros of random polynomials and their higher derivatives},
   JOURNAL = {Trans. Amer. Math. Soc.},
  FJOURNAL = {Transactions of the American Mathematical Society},
    VOLUME = {375},
      YEAR = {2022},
    NUMBER = {9},
     PAGES = {6311--6335},
      ISSN = {0002-9947,1088-6850},
   MRCLASS = {60G99 (30B20 30C15)},
  MRNUMBER = {4474893},
       DOI = {10.1090/tran/8674},
       URL = {https://doi.org/10.1090/tran/8674},
}

@article {Angst-Jurgen-Poly-Almostsurebehaviorofthecriticalpointsofrandom,
    AUTHOR = {Angst, J\"urgen and Malicet, Dominique and Poly, Guillaume},
     TITLE = {Almost sure behavior of the critical points of random
              polynomials},
   JOURNAL = {Bull. Lond. Math. Soc.},
  FJOURNAL = {Bulletin of the London Mathematical Society},
    VOLUME = {56},
      YEAR = {2024},
    NUMBER = {2},
     PAGES = {767--782},
      ISSN = {0024-6093,1469-2120},
   MRCLASS = {30C15 (60B10 60G57)},
  MRNUMBER = {4711583},
MRREVIEWER = {Sergey\ Berezin},
       DOI = {10.1112/blms.12963},
       URL = {https://doi.org/10.1112/blms.12963},
}

@article {Ghosh-numberofcomponents,
    AUTHOR = {Ghosh, Subhajit},
     TITLE = {On the number of components of random polynomial lemniscates},
   JOURNAL = {Electron. J. Probab.},
  FJOURNAL = {Electronic Journal of Probability},
    VOLUME = {29},
      YEAR = {2024},
     PAGES = {Paper No. 86, 24},
      ISSN = {1083-6489},
   MRCLASS = {60D05 (30C10)},
  MRNUMBER = {4761374},
       DOI = {10.1214/24-ejp1147},
       URL = {https://doi.org/10.1214/24-ejp1147},
}

@book {concentration-inequalities-BLMbook,
    AUTHOR = {Boucheron, St\'{e}phane and Lugosi, G\'{a}bor and Massart, Pascal},
     TITLE = {Concentration inequalities},
      NOTE = {A nonasymptotic theory of independence,
              With a foreword by Michel Ledoux},
 PUBLISHER = {Oxford University Press, Oxford},
      YEAR = {2013},
     PAGES = {x+481},
      ISBN = {978-0-19-953525-5},
   MRCLASS = {60E15 (60D05 60G15 60G50 60G70 62G20)},
  MRNUMBER = {3185193},
MRREVIEWER = {Sreenivasan Ravi},
       DOI = {10.1093/acprof:oso/9780199535255.001.0001},
       URL = {https://doi.org/10.1093/acprof:oso/9780199535255.001.0001},
}

@article {KLR,
    AUTHOR = {Krishnapur, Manjunath and Lundberg, Erik and Ramachandran,
              Koushik},
     TITLE = {Inradius of random lemniscates},
   JOURNAL = {J. Approx. Theory},
  FJOURNAL = {Journal of Approximation Theory},
    VOLUME = {299},
      YEAR = {2024},
     PAGES = {Paper No. 106018, 25},
      ISSN = {0021-9045,1096-0430},
   MRCLASS = {30C10 (31A15 60G60)},
  MRNUMBER = {4703355},
       DOI = {10.1016/j.jat.2024.106018},
       URL = {https://doi.org/10.1016/j.jat.2024.106018},
}

@article {crane-polynomialpreimage,
    AUTHOR = {Crane, Edward},
     TITLE = {The areas of polynomial images and pre-images},
   JOURNAL = {Bull. London Math. Soc.},
  FJOURNAL = {The Bulletin of the London Mathematical Society},
    VOLUME = {36},
      YEAR = {2004},
    NUMBER = {6},
     PAGES = {786--792},
      ISSN = {0024-6093,1469-2120},
   MRCLASS = {30C10 (26D05 30C85)},
  MRNUMBER = {2083754},
MRREVIEWER = {Albert\ Baernstein, II},
       DOI = {10.1112/S0024609304003509},
       URL = {https://doi.org/10.1112/S0024609304003509},
}

@book {Ahlfors,
    AUTHOR = {Ahlfors, Lars V.},
     TITLE = {Complex analysis},
    SERIES = {International Series in Pure and Applied Mathematics},
   EDITION = {Third},
      NOTE = {An introduction to the theory of analytic functions of one
              complex variable},
 PUBLISHER = {McGraw-Hill Book Co., New York},
      YEAR = {1978},
     PAGES = {xi+331},
      ISBN = {0-07-000657-1},
   MRCLASS = {30-01},
  MRNUMBER = {510197},
}

@article {Kabluchko-Seidel,
    AUTHOR = {Kabluchko, Zakhar and Seidel, Hauke},
     TITLE = {Distances between zeroes and critical points for random
              polynomials with i.i.d. zeroes},
   JOURNAL = {Electron. J. Probab.},
  FJOURNAL = {Electronic Journal of Probability},
    VOLUME = {24},
      YEAR = {2019},
     PAGES = {Paper No. 34, 25},
      ISSN = {1083-6489},
   MRCLASS = {60G99 (30B10 60B10 60G57)},
  MRNUMBER = {3940764},
MRREVIEWER = {Ofer\ Zeitouni},
       DOI = {10.1214/19-EJP295},
       URL = {https://doi.org/10.1214/19-EJP295},
}

@article {vonBahr-Esseen,
    AUTHOR = {von Bahr, Bengt and Esseen, Carl-Gustav},
     TITLE = {Inequalities for the {$r$}th absolute moment of a sum of
              random variables, {$1\leq r\leq 2$}},
   JOURNAL = {Ann. Math. Statist.},
  FJOURNAL = {Annals of Mathematical Statistics},
    VOLUME = {36},
      YEAR = {1965},
     PAGES = {299--303},
      ISSN = {0003-4851},
   MRCLASS = {60.20 (60.30)},
  MRNUMBER = {170407},
MRREVIEWER = {\v C.\ V.\ Stanojevi\'c},
       DOI = {10.1214/aoms/1177700291},
       URL = {https://doi.org/10.1214/aoms/1177700291},
}

@book {Petrov-sumsofindipendentrandomvariables,
    AUTHOR = {Petrov, V. V.},
     TITLE = {Sums of independent random variables},
    SERIES = {Ergebnisse der Mathematik und ihrer Grenzgebiete [Results in
              Mathematics and Related Areas]},
    VOLUME = {Band 82},
      NOTE = {Translated from the Russian by A. A. Brown},
 PUBLISHER = {Springer-Verlag, New York-Heidelberg},
      YEAR = {1975},
     PAGES = {x+346},
   MRCLASS = {60FXX (60GXX)},
  MRNUMBER = {388499},
}

@article {Bobkov-etal,
    AUTHOR = {Bobkov, S. G. and Chistyakov, G. P. and G\"otze, F.},
     TITLE = {Bounds for characteristic functions in terms of quantiles and
              entropy},
   JOURNAL = {Electron. Commun. Probab.},
  FJOURNAL = {Electronic Communications in Probability},
    VOLUME = {17},
      YEAR = {2012},
     PAGES = {no. 21, 9},
      ISSN = {1083-589X},
   MRCLASS = {60E10},
  MRNUMBER = {2943104},
       DOI = {10.1214/ECP.v17-2053},
       URL = {https://doi.org/10.1214/ECP.v17-2053},
}

@book {Hough-Krishnapur-Peres,
    AUTHOR = {Hough, J. Ben and Krishnapur, Manjunath and Peres, Yuval and
              Vir\'ag, B\'alint},
     TITLE = {Zeros of {G}aussian analytic functions and determinantal point
              processes},
    SERIES = {University Lecture Series},
    VOLUME = {51},
 PUBLISHER = {American Mathematical Society, Providence, RI},
      YEAR = {2009},
     PAGES = {x+154},
      ISBN = {978-0-8218-4373-4},
   MRCLASS = {60G55 (30B20 30C15 60B20 60F10 60G15 65H04 82B31)},
  MRNUMBER = {2552864},
MRREVIEWER = {Dmitry\ Beliaev},
       DOI = {10.1090/ulect/051},
       URL = {https://doi.org/10.1090/ulect/051},
}

@article {lundberg-ramachandran-JLMS17,
    AUTHOR = {Lundberg, Erik and Ramachandran, Koushik},
     TITLE = {The arc length and topology of a random lemniscate},
   JOURNAL = {J. Lond. Math. Soc. (2)},
  FJOURNAL = {Journal of the London Mathematical Society. Second Series},
    VOLUME = {96},
      YEAR = {2017},
    NUMBER = {3},
     PAGES = {621--641},
      ISSN = {0024-6107,1469-7750},
   MRCLASS = {60G60 (30C10 53C65)},
  MRNUMBER = {3742436},
       DOI = {10.1112/jlms.12086},
       URL = {https://doi.org/10.1112/jlms.12086},
}

@article {lerario-lundberg-PLMS16,
    AUTHOR = {Lerario, Antonio and Lundberg, Erik},
     TITLE = {On the geometry of random lemniscates},
   JOURNAL = {Proc. Lond. Math. Soc. (3)},
  FJOURNAL = {Proceedings of the London Mathematical Society. Third Series},
    VOLUME = {113},
      YEAR = {2016},
    NUMBER = {5},
     PAGES = {649--673},
      ISSN = {0024-6115,1460-244X},
   MRCLASS = {31A15 (14P25 30C10 30C15 60G60)},
  MRNUMBER = {3570241},
MRREVIEWER = {Malik\ Younsi},
       DOI = {10.1112/plms/pdw039},
       URL = {https://doi.org/10.1112/plms/pdw039},
}

@article {SVNagaevLargedeviationsofsumsofrandomvariables,
    AUTHOR = {Nagaev, S. V.},
     TITLE = {Large deviations of sums of independent random variables},
   JOURNAL = {Ann. Probab.},
  FJOURNAL = {The Annals of Probability},
    VOLUME = {7},
      YEAR = {1979},
    NUMBER = {5},
     PAGES = {745--789},
      ISSN = {0091-1798,2168-894X},
   MRCLASS = {60F10 (60F05 60F15 60G50)},
  MRNUMBER = {542129},
MRREVIEWER = {Werner\ Wolf},
       URL =
              {http://links.jstor.org/sici?sici=0091-1798(197910)7:5<745:LDOSOI>2.0.CO;2-S&origin=MSN},
}

@misc{cline-hsing,
      title={Large Deviation Probabilities for Sums of Random Variables with Heavy or Subexponential Tails}, 
      author={Daren B. H. Cline and Tailen Hsing},
      year={2022},
      eprint={2211.16340},
      archivePrefix={arXiv},
      primaryClass={math.PR},
      url={https://arxiv.org/abs/2211.16340}, 
}

@incollection {PemantleRivin,
    AUTHOR = {Pemantle, Robin and Rivin, Igor},
     TITLE = {The distribution of zeros of the derivative of a random
              polynomial},
 BOOKTITLE = {Advances in combinatorics},
     PAGES = {259--273},
 PUBLISHER = {Springer, Heidelberg},
      YEAR = {2013},
      ISBN = {978-3-642-30978-6; 978-3-642-30979-3},
   MRCLASS = {60G99 (30C10 30C15 60B10)},
  MRNUMBER = {3363974},
MRREVIEWER = {Dmitry\ Beliaev},
}

@article {KabluchkoCriticalpointsofrandompolynomia,
    AUTHOR = {Kabluchko, Zakhar},
     TITLE = {Critical points of random polynomials with independent
              identically distributed roots},
   JOURNAL = {Proc. Amer. Math. Soc.},
  FJOURNAL = {Proceedings of the American Mathematical Society},
    VOLUME = {143},
      YEAR = {2015},
    NUMBER = {2},
     PAGES = {695--702},
      ISSN = {0002-9939,1088-6826},
   MRCLASS = {30C10 (30C15 31A15)},
  MRNUMBER = {3283656},
MRREVIEWER = {Hiroshi\ Sekigawa},
       DOI = {10.1090/S0002-9939-2014-12258-1},
       URL = {https://doi.org/10.1090/S0002-9939-2014-12258-1},
}

@article {HaninBorisPairingofzeros,
    AUTHOR = {Hanin, Boris},
     TITLE = {Pairing of zeros and critical points for random polynomials},
   JOURNAL = {Ann. Inst. Henri Poincar\'e{} Probab. Stat.},
  FJOURNAL = {Annales de l'Institut Henri Poincar\'e{} Probabilit\'es et
              Statistiques},
    VOLUME = {53},
      YEAR = {2017},
    NUMBER = {3},
     PAGES = {1498--1511},
      ISSN = {0246-0203,1778-7017},
   MRCLASS = {30C10 (30C15 31A15 60G60)},
  MRNUMBER = {3689975},
MRREVIEWER = {Brian\ Simanek},
       DOI = {10.1214/16-AIHP767},
       URL = {https://doi.org/10.1214/16-AIHP767},
}

@article {Tao-sendov-conjecture,
    AUTHOR = {Tao, Terence},
     TITLE = {Sendov's conjecture for sufficiently-high-degree polynomials},
   JOURNAL = {Acta Math.},
  FJOURNAL = {Acta Mathematica},
    VOLUME = {229},
      YEAR = {2022},
    NUMBER = {2},
     PAGES = {347--392},
      ISSN = {0001-5962,1871-2509},
   MRCLASS = {30C15},
  MRNUMBER = {4554225},
MRREVIEWER = {Haseo\ Ki},
       DOI = {10.4310/acta.2022.v229.n2.a3},
       URL = {https://doi.org/10.4310/acta.2022.v229.n2.a3},
}

@article {Tao-Vu-Randommatrices-universality-circularlaw,
    AUTHOR = {Tao, Terence and Vu, Van},
     TITLE = {Random matrices: universality of {ESD}s and the circular law},
      NOTE = {With an appendix by Manjunath Krishnapur},
   JOURNAL = {Ann. Probab.},
  FJOURNAL = {The Annals of Probability},
    VOLUME = {38},
      YEAR = {2010},
    NUMBER = {5},
     PAGES = {2023--2065},
      ISSN = {0091-1798,2168-894X},
   MRCLASS = {60B20 (60F15 60F17)},
  MRNUMBER = {2722794},
MRREVIEWER = {Sasha\ Sodin},
       DOI = {10.1214/10-AOP534},
       URL = {https://doi.org/10.1214/10-AOP534},
}

@article {Kac-I,
    AUTHOR = {Kac, M.},
     TITLE = {On the average number of real roots of a random algebraic
              equation},
   JOURNAL = {Bull. Amer. Math. Soc.},
  FJOURNAL = {Bulletin of the American Mathematical Society},
    VOLUME = {49},
      YEAR = {1943},
     PAGES = {314--320},
      ISSN = {0002-9904},
   MRCLASS = {41.1X},
  MRNUMBER = {7812},
MRREVIEWER = {P.\ Erd\H os},
       DOI = {10.1090/S0002-9904-1943-07912-8},
       URL = {https://doi.org/10.1090/S0002-9904-1943-07912-8},
}

@article {Kac-II,
    AUTHOR = {Kac, M.},
     TITLE = {On the average number of real roots of a random algebraic
              equation. {II}},
   JOURNAL = {Proc. London Math. Soc. (2)},
  FJOURNAL = {Proceedings of the London Mathematical Society. Second Series},
    VOLUME = {50},
      YEAR = {1949},
     PAGES = {390--408},
      ISSN = {0024-6115},
   MRCLASS = {60.0X},
  MRNUMBER = {30713},
MRREVIEWER = {A.\ C.\ Offord},
       DOI = {10.1112/plms/s2-50.5.390},
       URL = {https://doi.org/10.1112/plms/s2-50.5.390},
}

@article {Rice-Mathematicalanalysisofrandomnoise,
    AUTHOR = {Rice, S. O.},
     TITLE = {Mathematical analysis of random noise},
   JOURNAL = {Bell System Tech. J.},
  FJOURNAL = {The Bell System Technical Journal},
    VOLUME = {23},
      YEAR = {1944},
     PAGES = {282--332},
      ISSN = {0005-8580},
   MRCLASS = {60.0X},
  MRNUMBER = {10932},
MRREVIEWER = {M.\ Kac},
       DOI = {10.1002/j.1538-7305.1944.tb00874.x},
       URL = {https://doi.org/10.1002/j.1538-7305.1944.tb00874.x},
}

@misc{berzin2022kacriceformulacontemporaryoverview,
      title={Kac-Rice formula: A contemporary overview of the main results and applications}, 
      author={Corinne Berzin and Alain Latour and José León},
      year={2022},
      eprint={2205.08742},
      archivePrefix={arXiv},
      primaryClass={math.CA},
      url={https://arxiv.org/abs/2205.08742}, 
}

@incollection {Linnik-Ontheprobabilityoflargedeviationsforthesumofindependentvariables,
    AUTHOR = {Linnik, Ju.\ V.},
     TITLE = {On the probability of large deviations for the sums of
              independent variables},
 BOOKTITLE = {Proc. 4th {B}erkeley {S}ympos. {M}ath. {S}tatist. and {P}rob.,
              {V}ol. {II}},
     PAGES = {289--306},
 PUBLISHER = {Univ. California Press, Berkeley-Los Angeles, Calif.},
      YEAR = {1960},
   MRCLASS = {60.30},
  MRNUMBER = {137142},
MRREVIEWER = {W.\ Hoeffding},
}

@article {Heyde-Onlargedeviationprobabilitiesinthecaseofattractiontoanon-normalstablelaw,
    AUTHOR = {Heyde, C. C.},
     TITLE = {On large deviation probabilities in the case of attraction to
              a non-normal stable law},
   JOURNAL = {Sankhy\=a{} Ser. A},
  FJOURNAL = {Sankhy\=a{} (Statistics). The Indian Journal of Statistics.
              Series A},
    VOLUME = {30},
      YEAR = {1968},
     PAGES = {253--258},
      ISSN = {0581-572X},
   MRCLASS = {60.30},
  MRNUMBER = {240854},
MRREVIEWER = {H.\ Kesten},
}

@article {Heyde-Onlargedeviationproblemsforsumsofrandomvariableswhicharenotattractedtothenormallaw,
    AUTHOR = {Heyde, C. C.},
     TITLE = {On large deviation problems for sums of random variables which
              are not attracted to the normal law},
   JOURNAL = {Ann. Math. Statist.},
  FJOURNAL = {Annals of Mathematical Statistics},
    VOLUME = {38},
      YEAR = {1967},
     PAGES = {1575--1578},
      ISSN = {0003-4851},
   MRCLASS = {60.30},
  MRNUMBER = {221564},
MRREVIEWER = {D.\ Siegmund},
       DOI = {10.1214/aoms/1177698712},
       URL = {https://doi.org/10.1214/aoms/1177698712},
}

@article {Heyde-Acontributiontothetheoryoflargedeviationsforsumsofindependentrandomvariables,
    AUTHOR = {Heyde, C. C.},
     TITLE = {A contribution to the theory of large deviations for sums of
              independent random variables},
   JOURNAL = {Z. Wahrscheinlichkeitstheorie und Verw. Gebiete},
  FJOURNAL = {Zeitschrift f\"ur Wahrscheinlichkeitstheorie und Verwandte
              Gebiete},
    VOLUME = {7},
      YEAR = {1967},
     PAGES = {303--308},
   MRCLASS = {60.30},
  MRNUMBER = {216549},
MRREVIEWER = {H.\ Kesten},
       DOI = {10.1007/BF00535016},
       URL = {https://doi.org/10.1007/BF00535016},
}

@article {Nagaev-II,
    AUTHOR = {Nagaev, A. V.},
     TITLE = {Integral limit theorems with regard to large deviations when
              {C}ram\'er's condition is not satisfied. {II}},
   JOURNAL = {Teor. Verojatnost. i Primenen.},
  FJOURNAL = {Akademija Nauk SSSR. Teorija Verojatnoste\u i\ i ee
              Primenenija},
    VOLUME = {14},
      YEAR = {1969},
     PAGES = {203--216},
      ISSN = {0040-361x},
   MRCLASS = {60.30},
  MRNUMBER = {247652},
MRREVIEWER = {J.\ L.\ Denny},
}

@article {Nagaev-I,
    AUTHOR = {Nagaev, A. V.},
     TITLE = {Integral limit theorems with regard to large deviations when
              {C}ram\'er's condition is not satisfied. {I}},
   JOURNAL = {Teor. Verojatnost. i Primenen.},
  FJOURNAL = {Akademija Nauk SSSR. Teorija Verojatnoste\u i\ i ee
              Primenenija},
    VOLUME = {14},
      YEAR = {1969},
     PAGES = {51--63},
      ISSN = {0040-361x},
   MRCLASS = {60.30},
  MRNUMBER = {247651},
}

@article {Mikosch-Nagaev-Largedeviationsofheavytailedsumswithapplicationsininsurance,
    AUTHOR = {Mikosch, T. and Nagaev, A. V.},
     TITLE = {Large deviations of heavy-tailed sums with applications in
              insurance},
   JOURNAL = {Extremes},
  FJOURNAL = {Extremes. Statistical Theory and Applications in Science,
              Engineering and Economics},
    VOLUME = {1},
      YEAR = {1998},
    NUMBER = {1},
     PAGES = {81--110},
      ISSN = {1386-1999,1572-915X},
   MRCLASS = {60F10 (60G50 60G70 62P05)},
  MRNUMBER = {1652936},
MRREVIEWER = {Boualem\ Djehiche},
       DOI = {10.1023/A:1009913901219},
       URL = {https://doi.org/10.1023/A:1009913901219},
}

@book {Mikosch-Embrechts-Claudia,
    AUTHOR = {Embrechts, Paul and Kl\"uppelberg, Claudia and Mikosch,
              Thomas},
     TITLE = {Modelling extremal events},
    SERIES = {Applications of Mathematics (New York)},
    VOLUME = {33},
      NOTE = {For insurance and finance},
 PUBLISHER = {Springer-Verlag, Berlin},
      YEAR = {1997},
     PAGES = {xvi+645},
      ISBN = {3-540-60931-8},
   MRCLASS = {60G70 (60Fxx 60K05 62M10 62P05)},
  MRNUMBER = {1458613},
MRREVIEWER = {R.\ A.\ Maller},
       DOI = {10.1007/978-3-642-33483-2},
       URL = {https://doi.org/10.1007/978-3-642-33483-2},
}

@article {Denisov-Dieker-Shneer-Largedeviationsforrandomwalksundersubexponentiality:thebig-jumpdomain,
    AUTHOR = {Denisov, D. and Dieker, A. B. and Shneer, V.},
     TITLE = {Large deviations for random walks under subexponentiality: the
              big-jump domain},
   JOURNAL = {Ann. Probab.},
  FJOURNAL = {The Annals of Probability},
    VOLUME = {36},
      YEAR = {2008},
    NUMBER = {5},
     PAGES = {1946--1991},
      ISSN = {0091-1798,2168-894X},
   MRCLASS = {60G50 (60F10)},
  MRNUMBER = {2440928},
       DOI = {10.1214/07-AOP382},
       URL = {https://doi.org/10.1214/07-AOP382},
}

@incollection {Nguyen-Van-Smallballprobabilityinversetheoremsandapplications,
    AUTHOR = {Nguyen, Hoi H. and Vu, Van H.},
     TITLE = {Small ball probability, inverse theorems, and applications},
 BOOKTITLE = {Erd\"os centennial},
    SERIES = {Bolyai Soc. Math. Stud.},
    VOLUME = {25},
     PAGES = {409--463},
 PUBLISHER = {J\'anos Bolyai Math. Soc., Budapest},
      YEAR = {2013},
      ISBN = {978-963-9453-18-0; 978-3-642-39285-6},
   MRCLASS = {60E15 (60C05 60F10)},
  MRNUMBER = {3203607},
MRREVIEWER = {Sho\ Matsumoto},
       DOI = {10.1007/978-3-642-39286-3\_16},
       URL = {https://doi.org/10.1007/978-3-642-39286-3_16},
}

@book {Aaronson-Jon-Anintroductiontoinfiniteergodictheory,
    AUTHOR = {Aaronson, Jon},
     TITLE = {An introduction to infinite ergodic theory},
    SERIES = {Mathematical Surveys and Monographs},
    VOLUME = {50},
 PUBLISHER = {American Mathematical Society, Providence, RI},
      YEAR = {1997},
     PAGES = {xii+284},
      ISBN = {0-8218-0494-4},
   MRCLASS = {28Dxx (58F11 58F17)},
  MRNUMBER = {1450400},
MRREVIEWER = {Cesar\ E.\ Silva},
       DOI = {10.1090/surv/050},
       URL = {https://doi.org/10.1090/surv/050},
}

@article {Kamae-Keane-Asimpleproofoftheratioergodictheorem,
    AUTHOR = {Kamae, Teturo and Keane, Michael},
     TITLE = {A simple proof of the ratio ergodic theorem},
   JOURNAL = {Osaka J. Math.},
  FJOURNAL = {Osaka Journal of Mathematics},
    VOLUME = {34},
      YEAR = {1997},
    NUMBER = {3},
     PAGES = {653--657},
      ISSN = {0030-6126},
   MRCLASS = {28D05},
  MRNUMBER = {1613108},
MRREVIEWER = {Nathaniel\ F. G. Martin},
       URL = {http://projecteuclid.org/euclid.ojm/1200787661},
}

\end{document}